\documentclass{CSML}

\def\dOi{11(3:18)2015}
\lmcsheading%
{\dOi}
{1--51}
{}
{}
{Feb.~23, 2014}
{Sep.~22, 2015}
{}

\ACMCCS{[{\bf Theory of computation}]: Logic}

\usepackage{amsfonts}
\usepackage{amsmath}
\usepackage{amsthm}
\usepackage{amssymb}
\usepackage[mathcal]{euscript}
\usepackage{bbm}
\usepackage{stmaryrd}
\usepackage{mathrsfs}

\usepackage{hyperref}
\usepackage[all]{xy}
\xyoption{2cell}
\xyoption{curve}
\UseAllTwocells
\SelectTips{cm}{}

\theoremstyle{plain}
\newtheorem{theorem}{Theorem}[section]
\newtheorem{lemma}[theorem]{Lemma}
\newtheorem{proposition}[theorem]{Proposition}
\newtheorem{corollary}[theorem]{Corollary}

\theoremstyle{definition}
\newtheorem{definition}[theorem]{Definition}
\newtheorem{example}[theorem]{Example}
\newtheorem{remark}[theorem]{Remark}

\newcommand{\Alg}{\mathsf{Alg}}
\newcommand{\BA}{\mathsf{BA}}

\newcommand{\Pos}{\mathsf{Pos}}
\newcommand{\DL}{\mathsf{DL}}
\newcommand{\Coalg}{\mathsf{Coalg}}

\newcommand{\Set}{\mathsf{Set}}

\newcommand{\Two}{\mathbbm{2}}

\newcommand{\pcal}{\mathcal P}

\def\Alg{{\mathsf{Alg}}}

\def\op{{\mathsf{op}}}

\newcommand{\ff}{\mathit{ff}}

\def\eps{\varepsilon}

\def\Lan{{\mathrm{Lan}}}

\def\wh#1{\widehat{#1}}

\newcommand{\opp}{{\mathsf{op}}}

\newcommand{\id}{\mathrm{id}}
\newcommand{\Id}{\mathsf{Id}}

\newcommand{\Pre}{\mathsf{Preord}}

\newcommand{\atomprop}{{\mathsf{AtProp}}}

\newcommand{\acal}{\mathscr{A}}

\newcommand{\lcal}{\mathcal{L}}
\def\K{\mathscr{K}}

\newcommand{\sem}[1]{[\![#1]\!]}

\begin{document}

%========================================================%

\title[Positive Fragments of Coalgebraic Logics]{Positive Fragments of Coalgebraic Logics\rsuper*}

%========================================================%

\author[A.~Balan]{Adriana Balan\rsuper a}
\address{{\lsuper a}University Politehnica of Bucharest, Romania}
\email{adriana.balan@mathem.pub.ro}
\thanks{{\lsuper a}Adriana Balan was supported by CNCSIS project
PD-19/03.08.2010.}

\author[A.~Kurz]{Alexander Kurz\rsuper b}
\address{{\lsuper b}University of Leicester, United Kingdom}
\email{ak155@mcs.le.ac.uk}
\author[J.~Velebil]{Ji\v{r}\'{\i} Velebil\rsuper c}
\address{{\lsuper c}Faculty of Electrical Engineering, Czech Technical University
         in Prague, Czech Republic}
\email{velebil@math.feld.cvut.cz}
 \thanks{{\lsuper c}Ji\v{r}\'{\i} Velebil is 
supported by  the grant No.~P202/11/1632
of the Czech Science Foundation.}

%========================================================%

\titlecomment{{\lsuper *}This is an extended and improved version of~\cite{calco2013}.}
% }

%========================================================%

\begin{abstract}
\noindent Positive modal logic was introduced in an influential 1995 paper of Dunn as the positive fragment of standard modal logic. His completeness result consists of an axiomatization that derives all modal formulas that are valid on all Kripke frames and are built only from atomic propositions, conjunction, disjunction, box and diamond.

\medskip\noindent 
In this paper, we provide a coalgebraic analysis of this theorem, which not only gives a conceptual proof based on duality theory, but also generalizes Dunn's result from Kripke frames to coalgebras for weak-pullback preserving functors. 

\medskip\noindent 
To facilitate this analysis we prove a number of category theoretic results on functors on the categories $\Set$ of sets and $\Pos$ of posets:

\medskip\noindent 
Every functor $\Set\to\Pos$ has a $\Pos$-enriched left Kan extension $\Pos\to\Pos$. Functors arising in this way are said to have a \emph{presentation in discrete arities}. In the case that $\Set\to\Pos$ is actually $\Set$-valued, we call the corresponding left Kan extension $\Pos\to\Pos$ its \emph{posetification}.

\medskip\noindent 
A set functor preserves weak pullbacks if and only if its posetification preserves exact squares. A $\Pos$-functor with a presentation in discrete arities preserves surjections.

\medskip\noindent 
The inclusion $\Set\to\Pos$ is dense. A functor $\Pos\to\Pos$ has a presentation in discrete arities if and only if it preserves coinserters of `truncated nerves of posets'. 
%
%\medskip\noindent 
A functor $\Pos\to\Pos$ is a posetification if and only if it preserves coinserters of truncated nerves of posets and discrete posets. 

\medskip\noindent 
A locally monotone endofunctor of an ordered variety has a presentation by monotone operations and equations if and only if it preserves $\Pos$-enriched sifted colimits.

\end{abstract}

%========================================================%

\maketitle

%========================================================%

\setcounter{tocdepth}{1}
\newpage
\tableofcontents

%========================================================%

\section{Introduction}

\medskip\noindent
Consider modal logic as given by atomic propositions, Boolean
operations, and a unary box, together with its usual axiomatisation stating that box preserves finite meets. In~\cite{dunn:pml}, Dunn answered the
question of an axiomatisation of the
positive fragment of this logic, where the positive fragment is given
by atomic propositions, lattice operations, and unary box and diamond (but no negation).

\medskip\noindent
Here we seek to generalize this result from Kripke frames to coalgebras
for a weak pullback preserving functor. Whereas Dunn had no need to
justify that the positive fragment actually \emph{adds} a modal
operator (the diamond), the general situation requires a conceptual
clarification of this step. And, as it turns out, what looks innocent
enough in the familiar case is at the heart of the general
construction.

\medskip\noindent 
In the general case, we start with a functor $T:\Set\to\Set$. From $T$ we can obtain by duality a functor $L:\BA\to\BA$ on the category $\BA$
of Boolean algebras, so that the free $L$-algebras are exactly the
Lindenbaum algebras of the modal logic. We are going to take the functor $L$ itself as the category theoretic counterpart of the corresponding modal logic. How should we construct the positive $T$-logic? Dunn gives us a hint in that he notes that in the
same way as standard modal logic is given by algebras over $\BA$,
positive modal logic is given by algebras over the category $\DL$ of
(bounded) distributive lattices. It follows that the positive fragment
of (the logic corresponding to) $L$ should be a functor $L':\DL\to\DL$
which, in turn, by duality, should arise from a functor
$T':\Pos\to\Pos$ on the category $\Pos$ of posets and monotone maps.

\medskip\noindent
The centrepiece of our construction is now the observation that any functor $T:\Set\to\Set$ has a universal extension to a functor
$T':\Pos\to\Pos$. Theorem~\ref{mainthm} then shows that this
construction $T\mapsto T' \mapsto L'$ indeed gives the positive
fragment of $L$ and so generalizes Dunn's theorem.

\medskip\noindent
An important observation about the positive fragment is the following:
given any Boolean formula, we can rewrite it as a positive formula
with negation appearing only on atomic propositions. In other words,
the translation $\beta$ from positive logic to Boolean logic given
by 
\begin{gather}
\label{eq:translation1}
\beta(\Diamond \phi)  =   \neg\Box\neg \beta(\phi)\\ 
\label{eq:translation2}
\beta(\Box \phi)  =  \Box\beta(\phi) 
\end{gather}
induces a bijection (on equivalence classes of formulas taken up to
logical equivalence). More algebraically, we can formulate this as
follows. 

\medskip\noindent
Given a Boolean algebra $B\in\BA$, let $LB$ be the free Boolean algebra
generated by the set $\{\Box b \mid b\in B\}$ modulo the axioms of modal
logic. Given a distributive lattice $A$, let $L'A$ be the free
distributive lattice generated by $\{\Box a\mid a\in
A\}\cup\{\Diamond a\mid a\in A\}$ modulo the axioms of positive
modal logic. Further, let us denote by $W:\BA\to\DL$ the forgetful
functor. Then the above observation that every modal formula can be
written, up to logical equivalence, as a positive modal formula with
negations pushed to atoms, can be condensed into the statement that
the (natural) distributive lattice homomorphism
\begin{equation}\label{eq:beta}
\beta_B:L'WB \to WLB
\end{equation}
induced by~Equations~\eqref{eq:translation1}-\eqref{eq:translation2} is an
isomorphism.

\medskip\noindent
Our main results, presented in Sections~\ref{sec:poscoalglog} and \ref{sec:monotone_predicate_liftings}, are the following. If $T'$ is an extension of $T$ and
$L,L'$ are the induced logics, then $\beta:L'W\to WL$ exists. If,
moreover, $T'$ is the induced extension (posetification) of $T$ and $T$ preserves weak pullbacks, then
$\beta$ is an isomorphism (Theorem~\ref{mainthm}). Furthermore, in the same way as the induced logic $L$ can be seen as the logic of all predicate liftings of $T$, the induced logic $L'$ is the logic of all \emph{monotone} predicate liftings of $T$ (Theorem~\ref{thm:mon-lift}). 

\medskip\noindent 
These results depend crucially on the fact that the posetification $T'$ of $T$ is defined as a completion with respect to $\Pos$-enriched colimits. We devote Section~\ref{sec:Pos-functors} to establishing some results on posetifications used later. We show that posetifications always exists (Theorem~\ref{thm:posetification}). Moreover, we characterize those functors $\Pos\to\Pos$ that arise as posetifications as the functors that preserve coinserters of `truncated nerves of posets' and discrete posets (Theorem~\ref{thm:pres-nerves}). We also establish properties of posetifications needed in Section~\ref{sec:poscoalglog}, for example, that a functor $\Set\to\Set$ preserves weak pullbacks if and only if its posetification preserves exact squares (Theorem~\ref{thm:wpb_exsq}).

\medskip\noindent 
On the algebraic side, the move to $\Pos$-enriched colimits guarantees that the modal operations are monotone. In Section~\ref{ak}, and recalling~\cite[Theorem 4.7]{kurz-rosicky:strong} 
stating that a functor  $L':\DL\to\DL$ preserves ordinary sifted colimits if and only if it has a presentation by operations and equations, we show that $L':\DL\to\DL$ preserves \emph{enriched} sifted colimits if and only if it has a presentation by \emph{monotone} operations and equations (Theorem~\ref{thm:presentmonop}). To see the relevance of a presentation result specific to monotone operations, observe that in the example of positive modal logic it is indeed the case that both $\Box$ and $\Diamond$ are monotone.

\medskip\noindent
From the point of view of category theory the results of Sections~\ref{sec:Pos-functors} and \ref{ak} are of independent interest. In addition to what we already discussed, we introduce the concept of functors $\Pos\to\Pos$ with presentations in discrete arities. They generalise posetifications and are functors that arise as left Kan extensions of functors $H:\Set\to\Pos$ along the discrete functor $D:\Set\to\Pos$. They are characterised as those functors preserving coinserters of `truncated nerves of posets' (Theorem~\ref{thm:discretearities}).  An important property of $\Pos$-functors with presentations in discrete arities is that---like $\Set$-functors but unlike general $\Pos$-functors---they preserve surjections (Proposition~\ref{prop:T'-pres-monot-surj}).

%========================================================%

\medskip\noindent {\textbf{Acknowledgments}}. The authors would like to thank the referees for their valuable suggestions.

%========================================================%

\section{A review of  coalgebraic logic}
\label{sec:coalgebras}

\medskip\noindent
A Kripke model $(W,R,v)$ (see e.g. \cite{brv:ml} for an introduction to modal logic) with $R\subseteq W\times W$ and $v:W\to
2^\atomprop$ can also be described as a coalgebra $W\to \pcal W \times
2^\atomprop$, where $\pcal W$ stands for the powerset of $W$. This
point of view suggests to generalize modal logic from Kripke frames to
coalgebras 
\[
\xi:X\to TX
\]
where $T$ may now be any functor $T:\Set\to\Set$, see \cite{rutten:uc} for an introduction. We  recover
Kripke models by putting $TX=\pcal X \times 2^\atomprop$. We also
recover the so-called bounded morphisms or p-morphisms as coalgebras
morphisms $f:(X, \xi)\to (X', \xi')$, that is, as maps $f:X\to X'$ such that 
$Tf\circ \xi = \xi'\circ f$. 

%========================================================%

\subsection{Coalgebras and algebras} 

More generally, for any category $\K$ 
and any functor $T:\K\to\K$, we have the
category $\Coalg(T)$ of $T$-coalgebras with objects and morphisms as
above. Dually, $\Alg(T)$ is the category where the objects
$\alpha:TX\to X$ are arrows in $\K$ and where the morphisms
$f:(X,\alpha)\to(X',\alpha')$ are arrows $f:X\to X'$ in $\K$ such
that $f\circ\alpha=\alpha'\circ Tf$. It is worth noting that
$T$-coalgebras over $\K$ are dual to $T^\op$-algebras over
$\K^\op$, that is, 
$\Coalg(T)^\op = \Alg(T^\op)$. Here $\K^\op$ is the category which has the same objects and arrows as $\K$ but domain and codomain of arrows interchanged and $T^\op:\K^\op\to \K^\op$ is the functor that has the same action on objects and morphisms as $T$.

%========================================================%

\subsection{Duality of Boolean algebras and sets} 

The abstract duality between algebras and coalgebras becomes particularly interesting if we put it on top of a concrete duality, such as the
dual adjunction between the category $\Set$ of sets and functions and the
category $\BA$ of Boolean algebras. We denote by $P:\Set^\op\to\BA$ the
functor taking powersets and by $S:\BA\to\Set^\op$ the
functor taking ultrafilters. Alternatively, we can
describe these functors by $PX=\Set(X,2)$ and $SA=\BA(A,\Two)$,
which also determines their action on arrows (here $\Two$ denotes the
two-element Boolean algebra). $P$ and $S$ are adjoint, satisfying
$\Set(X,SA)\cong\BA(A,PX)$. Restricting $S$ and $P$ to finite
Boolean algebras/sets, this adjunction becomes a dual 
equivalence~\cite[\textsection~VI.(2.3)]{johnstone:stone-spaces}.

%========================================================%

\subsection{Boolean logics for coalgebras, syntax} 

What now are logics for coalgebras? We follow a well-established methodology in
modal logic~\cite{brv:ml} and study modal logics via the associated
category $\acal$ 
of modal algebras. In modal logic, major milestones are \cite{jonnsontarski} and \cite{goldblatt}. In computer science, we have the seminal  work on domain theory in logical form \cite{abramsky:dtlf,abramskyjung}. 

\medskip\noindent More formally, given a modal logic $\lcal$
extending Boolean propositional logic with its category
$\acal$ of modal algebras, we describe $\lcal$ by a
functor $$L:\BA\to\BA$$ so that the category $\Alg(L)$ of algebras for
the functor $L$ coincides with $\acal$. In particular, the Lindenbaum
algebra of $\lcal$ will be the initial $L$-algebra.

\begin{example}\label{exle:Lpcal}
Let $L:\BA\to\BA$ be the functor mapping an algebra $A$ to the algebra $LA$ generated by $\Box a$, $a \in A$, and quotiented by the relation stipulating that $\Box$ preserves finite meets, that is,
\begin{equation}\label{equ:K}
\Box \top = \top \quad\quad\quad \Box(a\wedge b) = \Box a \wedge \Box b
\end{equation} 
$\Alg(L)$ is the category of modal algebras (Boolean algebras with operators), a result which goes back to~\cite{Abramsky88}.
\end{example}

%========================================================%

\subsection{Boolean logics for coalgebras, semantics}

The semantics of such a logic is described by a natural transformation
\[
\delta:LP\to PT^\op
\]
Intuitively, each modal operator in $LPX$ is assigned its meaning as a
subset of $TX$. More formally, $\delta$ allows us to lift
$P:\Set^\op \to\BA$ to a functor  
\begin{align*}
 \Coalg(T) \ & \overset{P^\sharp}{\to} \ \Alg(L) \\
(X,\xi) \ \ & \mapsto \ P^\sharp (X, \xi)=(PX, \ LPX\stackrel{\delta_X}{\longrightarrow} PTX\stackrel{P\xi}{\longrightarrow} PX)
\end{align*}
If we consider a formula $\phi$ to be an element of the initial $L$-algebra (the Lindenbaum algebra of the logic), then the semantics
$$\sem{\phi}_{(X,\xi)}$$ of $\phi$ as a subset of a coalgebra $(X,\xi)$ is given by the unique arrow from that
initial algebra to $ P^\sharp (X,\xi)$.

\begin{remark}
This account of the semantics of modal logic is typical for the coalgebraic approach: One first defines a one-step semantics $\delta$ mapping formulas with precisely one layer of modal operators (as described by $L$) to one-step-transitions on the semantic side (as described by $T$). Then one uses (co)induction to extend the `one-step situation' to arbitrary formulas and behaviors.
\end{remark}

\begin{example}\label{exle:deltaLpcal}
For the logic $L:\BA\to\BA$ from Example~\ref{exle:Lpcal} and the powerset functor $\pcal:\Set\to\Set$ 
we define the (one-step) semantics 
$\delta_X: L P X\to P \pcal ^\op X$ by
\begin{equation}\label{equ:delta}
  \Box a \mapsto \{b\subseteq X\mid b\subseteq a\}, \mbox{ for } a\in PX.
\end{equation} 
It is an old result in domain theory that $\delta_X$ is an isomorphism for finite $X$, see~\cite{Abramsky88}. This implies completeness of the axioms~\eqref{equ:K} with respect to Kripke semantics~\eqref{equ:delta}.
\end{example} 

%========================================================%

\subsection{Completeness and Expressiveness}\label{sec:Completeness} 
Two important properties of logics, with respect to their coalgebraic semantics, that can be discussed parametrically in $T$, $L$ and the semantics 
\[
\delta:LP\to PT^\op
\]
are completeness and expressiveness. The logic given by $(L,\delta)$ is \emph{complete} if $\sem{\phi}_{(X,\xi)}\subseteq\sem{\psi}_{(X,\xi)}$ for all coalgebras ${(X,\xi)}$ implies that $\phi\le\psi$ holds in the initial $L$-algebra. The logic is \emph{expressive} if for any two elements of a coalgebra  ${(X,\xi)}$ which are not bisimilar, there is a formula that separates them. 

\medskip\noindent In the duality approach to logics as exemplified by work such as Goldblatt's~\cite{goldblatt} in modal modal logic and Abramsky's~\cite{abramsky:dtlf} in domain theory, completeness and expressiveness follow immediately if $\delta$ is an isomorphism. Since in that work the modal algebras and the coalgebras are based on dually equivalent categories, such as Boolean algebras and Stone spaces (or distributive lattices and Priestley spaces), the requirement that $\delta$ is an isomorphism is reasonable.

\medskip\noindent In our setting, the situation is different since we only have an adjunction, not an equivalence, relating $\BA$ and $\Set$ (or $\DL$ and $\Pos$). One of the consequences is that $\delta$ typically fails to be an isomorphism and the best we can usually expect is to have $\delta_X$ an isomorphism for \emph{finite} $X$. But, as it turns out, under relatively mild conditions on the functors $T$ and $L$ involved, one can show that the logic $L$ is complete if and only if $\delta$ is injective~\cite{kupkeetal:algsem,kurz-petr:manysorted}, and that $L$ is expressive if the adjoint transpose $SL\to T^\op S$ of $\delta$ is injective~\cite{klin,jacobssokolova} or, in the case of $\Pos$, an embedding \cite{kkv:aiml}. We will pick up this discussion again at the end of \mbox{Sections~ \ref{sec:poscoalglog} and~\ref{sec:monotone_predicate_liftings}}. 

%========================================================%

\subsection{Functors having presentations by operations and equations}\label{sec:Presentations} 

Given that the notion of functor is axiomatic and rather abstract, the question arises which functors $L:\BA\to\BA$ can legitimately be considered to be a modal logic. For us, in this paper, the requirement on $L$ is that $L$ has a presentation by operations and equations \cite{bons-kurz:fossacs06}. We will discuss this notion in detail in Section~\ref{sec:Pos-functors}, for now it is enough to recall an example.

\medskip\noindent For a presentation $\langle\Sigma_\BA,E_\BA\rangle$ of $\BA$ we let 
$\Sigma_\BA=\{\bot,\top,\neg,\vee,\wedge\}$ and $E_\BA$ the usual equations of Boolean algebra.  For a presentation $\langle \Sigma_L, E_L\rangle$ of the functor $L:\BA\to\BA$ from Example~\ref{exle:Lpcal} we let $\Sigma_L=\{\Box\}$ and $E_L$ the equations  \eqref{equ:K}.

\medskip\noindent The first reason why the notion of a presentation of a functor is important to us is the following. If $\langle \Sigma_L, E_L\rangle$ is a presentation of $L$ by operations $\Sigma_L$ and equations $E_L$ and if $\langle \Sigma_\BA, E_\BA\rangle$ is a presentation
 of $\BA$ 
then $\Alg(L)$ is presented by 
$$\langle\Sigma_\BA+\Sigma_L,E_\BA+E_L\rangle.$$
In such a situation, the logical description of $L$-algebra is obtained in a modular way from the logical description of the base category $\BA$ and the logical description of the functor $L$. We call the elements of $\Sigma_L$ \emph{modal operators} and the elements of $E_L$ \emph{modal axioms} or just axioms.

\medskip\noindent The second reason is that the class of functors having presentations  can be captured in a purely semantic way: a functor $L$ has a presentation by operations and equations if and only if  $L$ is determined by its action on finitely generated free algebras, if and only if $L$ preserves sifted colimits, see~\cite{kurz-rosicky:strong}. 
 We shall give more 
details on sifted colimits in Section~\ref{sec:sifted} below.

\medskip\noindent Most succinctly, endofunctors of $\BA$ having presentations by operations and equations are precisely those that arise as 
{\em left Kan extensions\/} 
of their restrictions along the inclusion functor $\BA_{\ff}\to\BA$ 
of the full subcategory $\BA_\ff$ of $\BA$ consisting of free algebras
on finitely many generators.

%========================================================%

\subsection{The (finitary, Boolean) coalgebraic logic of a $\Set$-functor}\label{subsec:BAlogic}
The question considered in this paragraph is: given the coalgebra-type $T$, how can we define the logic $\mathbf L$ of $T$? As there are many different logics $L$ for $T$, we shall use $\mathbf L$ for \emph{the} logic of $T$, that is, the strongest logic that captures all aspects of $T$ that can be expressed by a finitary Boolean logic.

\medskip\noindent The general considerations laid out above suggest that in order to define the finitary (Boolean) coalgebraic logic associated 
to a given functor $T:\Set\to\Set$ one first puts
\begin{equation}\label{eq:LFn}
  \mathbf LFn = PT^\op SFn
\end{equation}
where $Fn$ denotes the free Boolean algebra over $n$ generators, 
for $n$ ranging over natural numbers. The functor 
$\mathbf{L}:\BA\to\BA$ is then defined as a left Kan extension.
The semantics $\boldsymbol \delta$ is given by observing that natural
transformations $\boldsymbol \delta :\mathbf LP\to PT^\op$ 
are in bijection with natural transformations 
\begin{equation}\label{eq:hatdelta}
\widehat{\boldsymbol\delta} : \mathbf L\to PT^\opp S
\end{equation} 
and we can let $\widehat{\boldsymbol\delta}$ 
to be the identity on finitely generated free algebras.

\medskip\noindent  More explicitly, $\mathbf LA$ can be represented as the 
free Boolean algebra on the set
\[
\{
\sigma(a_1,\ldots a_n) 
\mid 
\sigma \in PT^\op SFn, a_i \in A, n<\omega
\}
\]
modulo appropriate axioms. The semantics 
$\boldsymbol \delta_X:\mathbf LPX\to PT^\op X$ is 
given by putting 
$$
\boldsymbol \delta\sigma(a_1,\ldots a_n)=PT^\op(\hat a)(\sigma)
$$ 
where $\hat a:X\to SFn$ is the adjoint
transpose of $(a_1,\ldots a_n):n\to UPX$, with $F\dashv U:\BA\to\Set$ being the familiar adjunction.

\medskip\noindent It is worth noting that the elements in $PT^\op SFn$ 
are, by the Yoneda Lemma, 
in one-to-one correspondence with natural transformations
$\Set(-,2^n)\to\Set(T-,2)$. The latter natural 
transformations are also known as 
{\em predicate liftings\/}~\cite{Pattinson03}. 
Hence we see that the logic $\mathbf L$ coincides with 
the logic of all predicate liftings 
of~\cite{Schroder05}, with the difference that the functor 
$\mathbf L$ also incorporates a complete set of axioms. The axioms are 
important to us as otherwise the natural transformation 
$\beta$ mentioned in the introduction, see~Equation~\eqref{eq:beta}, 
might not exist.

\medskip\noindent Of course, in concrete
examples one is often able to obtain much more succinct presentations:

\begin{proposition}\label{prop:mathbf L} For $T$ the powerset functor, the functor $\mathbf L$ defined by~Equation~\eqref{eq:LFn} is isomorphic to the functor $L$ of Example~\ref{exle:Lpcal}. 
\end{proposition}

\begin{proof}
In analogy with~Equation~\eqref{eq:hatdelta}, let  
$\widehat{\delta}:L \to PT^\op S$ be the transpose of 
$\delta: L P\to P \pcal ^\op$ as defined in~Equation~\eqref{equ:delta}. We know from Example~\ref{exle:deltaLpcal} that $\hat\delta_{Fn}:LFn \to PT^\op SFn =\mathbf LFn$ is an isomorphism. But as both $L$ and $\mathbf L$ are determined by their action on finitely generated free algebras, this extends to an isomorphisn $L\to \mathbf L$.
\end{proof}

\begin{remark}
The functor $\mathbf L$ is universal in the sense that any other finitary Boolean coalgebraic logic $L$ for $T$ is uniquely determined by the natural transformation $L\to\mathbf L$ constructed in the proof above. More formally, we can express this universality as follows: denote by $\mathsf{Sift}[\BA, \BA]$ the category of sifted-colimit-preserving functors from $\BA$ to $\BA$ and by $I$ the inclusion $\mathsf{Sift}[\BA, \BA]\to[\BA, \BA]$. Then $\widehat{\boldsymbol\delta}:\mathbf L\to PT^\op S$ as given in~Equation~\eqref{eq:hatdelta} is a final object in the comma category $I/PT^\op S$. 

Proposition~\ref{prop:mathbf L} can then be understood as saying that the logic defined by finality as above has a simple concrete presentation given by~Equations~\eqref{equ:K}-\eqref{equ:delta}.
\end{remark}

%========================================================%

\subsection{Outlook: Positive coalgebraic logic.} \label{subsec:DLlogic}

It is evident that, at least for some of the developments above, not only the functor $T$, but also the categories $\Set$ and $\BA$ can
be considered to be parameters. Accordingly, one expects that positive coalgebraic logic takes place over the category $\DL$ of (bounded) distributive lattices which in turn, is part of an adjunction $P':\Pos^\opp\to\DL$, taking upsets, and $S':\DL\to\Pos^\opp$, taking prime filters, or, equivalently, $P'X=\Pos(X,\Two)$ and $S'A=\DL(A,\Two)$ where $\Two$ is, as before, the two-element chain (now considered, depending on the context, either as a poset or as a distributive lattice). Consequently, the `natural semantics' of positive logics is `ordered Kripke frames', or coalgebras over posets. 
One side of this argument is formal: coalgebras over $\Pos$ are to logics over $\DL$ what coalgebras over $\Set$ are to logics over $\BA$. Another side of the argument goes as follows: to provide the semantics for a logic without negation, we need to distinguish between sets and their complements. This is most easily done by stipulating that the semantics of formulas is given by upward closed sets with respect to some order, as the complements then are downward closed sets.

\medskip\noindent Replaying the developments above with $\Pos$ and $\DL$ instead of $\Set$ and $\BA$, we may define a logic for $T'$-coalgebras, with $T':\Pos\to\Pos$, to be given by a natural transformation
\begin{equation}\label{eq:Ldelta'}
\boldsymbol \delta': \mathbf L'P'\to P'T'^\op
\end{equation}
where $\mathbf L':\DL \to \DL$ is a functor 
determined \emph{by its action on finitely discretely generated 
free distributive lattices}, namely
\begin{equation}\label{eq:L'F'n}
  \mathbf L'F'Dn = P'T'^\op S'F'Dn
\end{equation}
and $\boldsymbol\delta'$ is given by its transpose 
in the same way as in~Equation~\eqref{eq:hatdelta}. 
Here, $D:\Set\to \Pos$ denotes the functor equipping 
a set with the discrete order, 
and $F':\Pos \to \DL$ is the functor mapping a poset to the free distributive lattice on it. 

\begin{example}\label{exle:L'pcal}

Given a poset $X$, a subset $Y\subseteq X$ is called \emph{convex} if $y\leq y'\leq y''$ and $y,y''\in Y$ imply $y'\in Y$. The 
{\em convex powerset functor\/} 
$\pcal ':\Pos\to \Pos$ maps a poset to the set of its convex subsets, ordered by the Egli-Milner order, and a monotone map to its direct image. Let now $L':\DL\to\DL$ be the functor mapping a distributive lattice $A$ to the distributive lattice $ L'A$ generated by $\Box a$ and $\Diamond a$ for all $a \in A$, and subject to the relations stipulating that $\Box$ preserves finite meets, $\Diamond$ preserves finite 
joins, and that the inequalities
\begin{equation*}\label{equ:PML}
\Box a \wedge \Diamond b \le \Diamond(a\wedge b) \quad\quad\quad 
\Box(a\vee b)\le \Diamond a \vee \Box b
\end{equation*}
hold.

The natural transformation $\delta'_X:L'P' X\to P' {\pcal'}^\op X$ is 
defined by putting
\begin{equation*}\label{equ:diamonddelta}
  \Diamond a \mapsto \{b\subseteq X\mid \mbox{ $b$ is a convex subset of $X$ and } b\cap a\not=\emptyset \}, \mbox{ for $a\in P'X$,}
\end{equation*} 
the clause for $\Box a$ being the same as in~Equation~\eqref{equ:delta}. 
\end{example}

\begin{remark}
$\Alg( L')$ is the category of positive modal algebras of 
Dunn~\cite{dunn:pml}. We shall later see in Corollary~\ref{cor:L'isboldL'} that it is isomorphic to $\Alg(\mathbf L')$. We have again that for a finite poset $X$, $\delta'_X$ is
an isomorphism, a representation first stated in~\cite{johnstone:stone-spaces,johnstone:vietoris-locales}, the connection with modal logic being given 
by~\cite{Abramsky88,robinson:powerdomain,winskel:powerdomain} and investigated from a coalgebraic point of view in~\cite{palmigiano:cmcs03-j}. As opposed to \cite{palmigiano:cmcs03-j}, we take the set-theoretic semantics of modal logic as fundamental.
We therefore do not have to use the topological semantics based on Stone or Priestley duality: all we need is contained in the adjunctions $S\dashv P:\Set^\op\to\BA$ and $S'\dashv P':\Pos^\op\to\DL$. 
\end{remark}

%========================================================%

\subsection{Outlook: Coalgebraic logic enriched over $\Pos$} 

Moving from ordinary categories to categories enriched over 
$\Pos$ plays a major role in this paper. 
The reason is that enrichment over $\Pos$ takes 
automatically care of the fact that positive modal logics extend 
the logic of distributive lattices by \emph{monotone} modal 
operations. This is crucial from the point of view of our main 
application, namely positive modal logic.  

\medskip\noindent Throughout the paper, however, we shall encounter many 
more reasons why to move to an enriched setting.
Some of the reasons are the following.
\begin{enumerate}
\item The category $\Pos$ is the cocompletion under \emph{enriched} sifted colimits of the category of finite sets.
\item The posetification, to be locally monotone, must be defined as an \emph{enriched} left Kan extension. 
\item Among all functors on $\Pos$, posetifications are characterized by `coinserters of truncated nerves' where a coinserter is the enriched analogue of a coequalizer.
\item In the ordered setting, one is frequently interested in definability by inequations ($\le$) instead of definability by equations: quotienting by inequations corresponds to taking a coinserter instead of a coequalizer.
\item Having a presentation by \emph{monotone} operations and equations in \emph{discrete arities} is characterized by preservation of \emph{enriched} sifted colimits.
\end{enumerate}

\section{A review of $\Pos$-enriched categories}
  
\medskip\noindent
Below we recall some notions of enriched category theory needed in the sequel. Most of this section is rather technical, but we have decided to include it in order to keep the paper self-contained. However, for more details, we refer the reader to Kelly's monograph~\cite{kelly:enriched}.

%========================================================%

\subsection{The category $\Pos$ of posets and monotone maps}  

The category $\Pos$ has partial orders (posets) as objects and monotone maps as arrows. $\Pos$ is complete and cocomplete (even locally finitely
presentable~\cite{adamekrosicky}). Limits are computed as in $\Set$, with the order on the limit being the largest relation making the maps in the cocone monotone.
Colimits are easiest to compute in two steps. First, colimits in the category of preorders are computed as in $\Set$, with the preorder on the colimit being the smallest one making the maps in the cocone monotone. Second, one quotients the preorder by anti-symmetry in order to obtain a poset (directed colimits, however, are computed as in $\Set$, see~\cite{adamekrosicky}). $\Pos$ is also cartesian closed, with the internal hom $[X,Y]$ being the poset of monotone maps from $X$ to $Y$, ordered pointwise.

%========================================================%

\subsection{$\Pos$-enriched categories}\label{sec:pos}

We shall consider categories {\em enriched\/} in $\Pos$. 
Thus, a $\Pos$-enriched category $\K$ is a category with ordered homsets, such that composition is monotone in both arguments:
$g\circ f\leq k\circ h$ whenever $g\leq k$ and $f\leq h$; a $\Pos$-enriched functor $T:\K\to \mathscr L$ is a 
{\em locally monotone\/} functor, 
that is, it preserves the order on the homsets: $f\leq g$ implies $Tf\leq Tg$. 
A $\Pos$-natural transformation between locally monotone functors is just a natural transformation, the extra condition of enriched naturality being vacuous here. 
Consequently, a $\Pos$-adjunction between two $\Pos$-categories is just an ordinary adjunction between two locally monotone functors.
The category of $\Pos$-enriched functors from $\K$ to $\mathscr L$ and natural transformations between them will be denoted by 
$[\K, \mathscr L]$. The {\em opposite category\/} 
$\K^\mathsf{op}$ of $\K$ has just the sense of morphisms reversed, the order on hom-posets remains unchanged. 
To avoid overloaded notation, we shall mostly employ the same notation for a $\Pos$-category $\K$ and its 
{\em underlying ordinary category\/}, unless it is necessary 
to emphasize the distinction between them. In such case, the ordinary category will be denoted by $\K_o$, and we proceed similarly for $\Pos$-functors. 

\medskip\noindent Besides $\Pos$ itself, an example of a $\Pos$-enriched category is $\Set$, the category of sets and functions, considered discretely enriched. In the chain 
$$
C\dashv D\dashv V:\Pos\rightarrow \Set
$$
of adjunctions  between the connected components functor, the discrete functor and the forgetful one, only the adjunction 
$C\dashv D:\Set\to\Pos$ is enriched; in particular the discrete functor $D:\Set\to \Pos$ is locally monotone, while the forgetful functor $V:\Pos \to \Set$ fails to be so. 

\medskip\noindent Observe also that, due to the discrete enrichment, 
{\em any\/} functor $\Set\to \Set$ is automatically locally monotone.

%========================================================%

\subsection{Weighted (co)limits; coinserters; Kan-extensions} 
\label{sec:3.C}
Recall from~\cite{kelly:enriched} that the proper concepts of limits 
and colimits in enriched category theory are those of {\em weighted\/}
(co)limits. Specifically, the {\em colimit\/} of a $\Pos$ functor 
$G:\mathscr D \to \K$ weighted by a $\Pos$-functor 
$\varphi:\mathcal {\mathscr{D}}^\mathsf{op} \to \Pos$ 
is an object $\varphi * G$ in $\K$, 
together with an isomorphism 
\[
\K(\varphi *G, X) 
\cong 
[{\mathscr{D}}^\op, \Pos](\varphi, \K(G-, X))
\] 
of posets, natural in $X$ in $\K$. 

\medskip\noindent Dually, a {\em limit\/} of $G:\mathscr D\to \K$ 
weighted by $\varphi:\mathscr D \to \Pos$ 
is an object $\{\varphi, G\}$ in $\K$, 
together with an isomorphism 
\[
\K(X, \{\varphi,G\}) 
\cong 
[\mathscr D, \Pos](\varphi,\K(X, G-))
\] 
of posets, again natural in $X$ in $\K$. 

\medskip\noindent One example of weighted colimits are copowers that arise from constant weights and diagrams. 
Specifically, the {\em copower\/} 
$
P\bullet X
$ 
of a poset $P$ in $\Pos$ and an object $X$ in $\K$ 
is characterized by the natural isomorphism 
$$
\K(P \bullet X, -) \cong \Pos(P, \K(X, -))
$$ 

\medskip\noindent Another example of weighted (co)limit that will 
later appear in the paper is the (co)inserter:

\begin{example}\cite{kelly:elem}
A \emph{coinserter} is a colimit that has weight 
$\varphi:\mathscr{D}^\op \to \Pos$, where $\mathscr D$ 
is the category
\[
\xymatrix{\cdot \ar@<0.5ex>[r] \ar@<-0.5ex>[r] & \cdot}
\] 
The category $\mathscr{D}$
gets mapped by $\varphi$ to the parallel pair 
\[
\xymatrix{\mathbbm 2 &\ar@<-0.5ex>[l]_1 \ar@<0.5ex>[l]^0 \mathbbm 1}
\]
in $\Pos$, with arrow $0$ mapping to $0\in\mathbbm 2$ and arrow $1$ mapping to $1\in\mathbbm 2$ (recall that $\mathbbm 2$ is the poset  $\{0\leq 1\}$). A functor $G$ from $\mathscr D$ 
to a $\Pos$-category $\K$ corresponds to a parallel 
pair $d^0, d^1: X\rightrightarrows Y $ of arrows in $\K$. 

In more detail, the coinserter of $d^0, d^1$ consists 
of an object $\mathsf{coins}(d^0,d^1)$ and an arrow 
$c:Y\to \mathsf{coins}(d^0,d^1)$, with $c \circ d^0\leq c \circ d^1$.
The object and the arrow have to satisfy the universal
property, of course. Due to the enrichment in posets,
the universal property has two aspects:
\begin{enumerate}
\item 
The $1$-dimensional aspect of universality asserts that, given 
any  $q:Y\to Z$ with $q \circ d^0\leq q \circ d^1$, there is 
a unique $h:\mathsf{coins}(d^0,d^1)\to Z$ with $h\circ c=q$.
\item
The $2$-dimensional aspect of universality ensures that 
the above assignment $q\mapsto h$ is monotone. That is, 
given $q,q':Y\to Z$ with $q\leq q'$, $q\circ d^0\leq q\circ d^1$ 
and $q'\circ d^0\leq q'\circ d^1$, the corresponding unique arrows $h,h':\mathsf{coins}(d^0,d^1)\to Z$ satisfy $ h\leq h'$.
\end{enumerate}

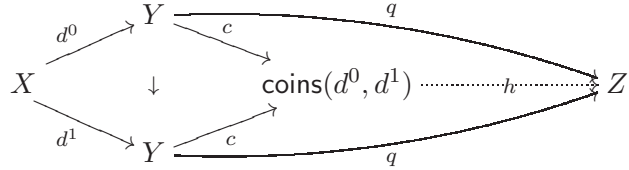
\[
\xymatrix@C=30pt@R=10pt{& Y \ar@{}[dd]|{\downarrow} \ar[dr]^-{c} \ar@/^2ex/[drrr]^q& \\ \qquad X \ \ar[ur]^{d^0} \ar[dr]_{d^1} && \mathsf{coins}(d^0,d^1) \ar@{..>}[rr]|-h && Z\\ & Y \ar[ur]_-{c} \ar@/_2ex/ [urrr]_q } 
\]

\noindent The coinserter is called \emph{reflexive} if $d^0$ and $d^1$ have a common right inverse $i:Y\to X$; that is, if the 
equalities $d^0\circ i= d^1 \circ i = \id_Y$ hold. 

By reversing the direction of the arrows, one obtains the dual notion of a (coreflexive) {\em inserter\/}. 
\end{example}

\medskip\noindent We shall later use also {\em co-comma objects\/}, which generalize coinserters, in the sense that this time $\mathscr D$ is 
\[
\xymatrix{
\cdot
&
\cdot 
\ar[r]
\ar[l]
&
\cdot
}
\] and is mapped by $\varphi$ to
\[
\xymatrix{
\mathbbm 1
&
\mathbbm 2
\ar[r]^1
\ar[l]_0
&
\mathbbm 1
}
\]
while a functor $G:\mathscr D \to \mathscr K$ corresponds to a pair of arrows with common domain $d^0:X \to Y$, $d^1:X \to Z$. More in detail, the co-comma object of $d^0$ and $d^1$ consists of an object $\mathsf{cocomma}(d^0,d^1)$, together with arrows $p:Y\to \mathsf{cocomma}(d^0,d^1)$, $q:Z \to  \mathsf{cocomma}(d^0,d^1)$, such that $p\circ d^0 \leq q \circ d^1$, as in the diagram below. We leave to the reader the explicit description of its universal property.
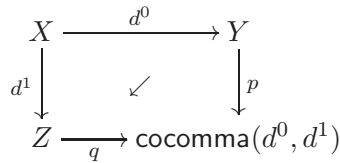
\[
\xymatrix{
X \ar[r]^{d^0} \ar[d]_{d^1} & Y \ar[d]^p \ar@{}[dl]|{\swarrow} \\
Z \ar[r]_-q &  \mathsf{cocomma}(d^0,d^1) }
\]
The dual notion is called the {\em comma object\/}.

\begin{remark}\label{rem:construction_coinserter} 
Informally speaking, whereas coequalizers are well-known to 
take quotients with respect to equivalence relations,
{\em coinserters take quotients with respect to preorders\/}. 

For later use, we recall how coinserters are built in $\Pos$. 
For a pair of monotone maps $d^0,d^1:X\to Y$, define first 
a binary relation $\mathbf r$ on the underlying set of the 
poset $Y$ as follows: given $y,y'\in Y$, say that $y\ \mathbf r\ y'$ if there are $x_0, \ldots , x_n\in X$ such that
\[
\xymatrix@C=7pt{& & x_0 \ar[dl]_{d^0} \ar[dr]^{d^1} &&&\ar[dl]_{d^0} \ar[dr]^{d^1} x_1 &&\ldots && x_n \ar[dl]_{d^0} \ar[dr]^{d^1} & & \\
y \ar@{}[r]|-{\leq} & d^0(x_0)  &&d^1(x_0)\ar@{}[r]|{\leq} &d^0(x_1)&&d^1(x_1) \ar@{}[r]|-{\leq} & \ldots \ar@{}[r]|-{\leq} & d^0(x_n) && d^1(x_n)& y' \ar@{}[l]|-{\leq} }
\]
It is easy to see that $\mathbf r$ is a reflexive and transitive relation, thus a preorder on $Y$. Then the coinserter of $d^0$ and $d^1$ is the quotient of $Y$ with respect to the equivalence relation induced by $\mathbf r$, with order $[y]\leq [y']$ if and only if $y \ \mathbf r\ y'$. 
\end{remark}

\medskip\noindent Reflexive coinserters play a similar role to reflexive coequalizers in ordinary category theory. For example, one can prove the following $\Pos$-enriched version of~\cite[Lemma~0.17]{johnstone:topos-theory} (see also~\cite[Section~2]{kelly-lack-walters:coinv}). 

\begin{lemma}($3\times 3$ lemma for coinserters) %\label{3x3coinserter}
Consider in a $\Pos$-category $\mathscr C$ the diagram 
\[
\xymatrix@R=35pt@C=35pt{X_1 \ar@<0.5ex>[r]^{f_1} \ar@<-0.5ex>[r]_{f_2} \ar@<0.5ex>[d]^{a_2} \ar@<-0.5ex>[d]_{a_1} & X_2 \ar@<0.5ex>[d]^{b_2} \ar@<-0.5ex>[d]_{b_1} \ar[r]^{f_3} & X_3 \ar@{..>}@<0.5ex>[d]^{c_2} \ar@{..>}@<-0.5ex>[d]_{c_1} \\
Y_1 \ar@<0.5ex>[r]^{g_1} \ar@<-0.5ex>[r]_{g_2} \ar[d]_{a_3} & Y_2 \ar[r]^{g_3} \ar[d]_{b_3} & Y_3 \ar@{..>}[d]_{c_3} 
\\
Z_1\ar@{..>}@<0.5ex>[r]^{h_1} \ar@{..>}@<-0.5ex>[r]_{h_2} & Z_2 \ar@{..>}[r]^{h_3} & Z_3 }
\]
where 
\begin{enumerate}
\item 
The first two rows and columns are coinserters.
\item The equalities below hold: 
\begin{equation*}\label{commut}
b_i \circ f_j = g_j \circ a_i \ (\ i,j=1,2\ ) \quad 
\end{equation*}
\end{enumerate}
These induce the arrows $c_1, c_2, h_1, h_2$ in an obvious way. Finally, let $h_3$ be the coinserter (assuming it exists in $\mathscr C$) of $h_1$ and $h_2$, and denote by $c_3:Y_3\to Z_3$ the induced unique mediating arrow. Then:
\begin{enumerate}
\item The last column is also a coinserter.

\item If additionally the first row and columns are reflexive coinserters, then the diagonal 
\[
\xymatrix@C=35pt{X_1 \ar@<0.35ex>[r]^{b_1\circ f_1%=g_1\circ a_1
}  \ar@<-0.35ex>[r]_{b_2\circ f_2%=g_2\circ a_2
} & Y_2 \ar[r]^{h_3\circ b_3% = h_3 \circ b_3
} & Z_3 }
\]
is again a coinserter, which is reflexive if the second row (column) is again a reflexive coinserter. 
\item Reflexivity of the first two rows and columns imply reflexivity of the remaining third row and column. 
\end{enumerate}
\end{lemma}

\begin{proof}
To see that $c_3$ is a coinserter, use first the 2-dimensional aspect of the coinserter $(X_3,f_3)$ to conclude $c_3 \circ c_1\leq c_3\circ c_2$. Next, given $w_1:Y_3\to W$ with $w_1\circ c_1\leq w_1\circ c_2$, it induces an arrow $w_2:Z_2\to W$ such that $w_2 \circ b_3 = w_1 \circ g_3$. Then the 2-dimensional part of the coinserter $(Z_1, a_3)$ yields $w_2\circ h_1 \leq w_2 \circ h_2$, thus it induces an arrow $w_3:Z_3\to W$ with $w_3 \circ h_3=w_2$. We have that 
\[w_3\circ c_3 \circ g_3= w_3 \circ h_3 \circ b_3 = w_2 \circ b_3 = w_1 \circ g_3
\]
and using that $g_3$ is an epimorphism we conclude $w_3 \circ c_3=w_1$. Finally, if $w_1, \bar w_1:Y_3\to W$ are such that $w_1\leq \bar w_1 $, $w_1\circ c_1\leq w_1\circ c_2$ and $\bar w_1\circ c_1\leq \bar w_1\circ c_2$, then successively we obtain $w_2 \leq \bar w_2$ and $w_3 \leq \bar w_2$ by using the 2-dimensional aspect of coinserters $(Z_2,b_3)$, respectively $(Z_3, h_3)$.

For the second part, denote by $i:X_2\to X_1$ and $j:Y_1 \to X_1$ the common right inverses of the parallel pairs of morphisms $f_1, f_2$, respectively $a_1, a_2$. Notice then that for an arrow $u_1:Y_2 \to U$ such that $u_1 \circ b_1 \circ f_1 \leq u_1 \circ b_2 \circ f_2$, precomposition with $i$ induces the inequality $u_1 \circ b_1 \leq u_1 \circ b_2$, thus we can find an arrow $u_2: Z_2 \to U$ with $u_2 \circ b_3= u_1$. In order to see that $u_2 \circ h_1 \leq u_2 \circ h_2$, use the first that precomposing 
\[
u_1 \circ g_1 \circ a_1 = u_1 \circ b_1 \circ f_1 \leq u_1 \circ b_2 \circ f_2= u_1 \circ g_2 \circ a_2
\] 
with $j$ yields $u_1 \circ g_1 \leq u_1 \circ g_2$, and next use the 2-dimensional aspect of the coinserter $(Z_1, a_3)$. From $u_2 \circ h_1 \leq u_2\circ  h_2$ we see that there is an arrow $u_3:Z_3\to U$ with $u_3 \circ h_3 = u_2$, thus $u_3 \circ h_3 \circ  b_3 = u_2 \circ b_3 = u_1$. The remaining 2-dimensional aspect of the requested coinserter can be easily proved along these lines, and we leave it to the reader, as well as the assertions on reflexivity.
\end{proof}

\medskip\noindent The importance of reflexive coinserters for us stems from the fact that each poset can be canonically expressed as a reflexive coinserter of \emph{discrete} posets:

\begin{proposition}\label{poset=refl_coins_discrete}
Let $X$ be a poset. Denote by $X_0$ its underlying \emph{set}, and by $X_1$ the \emph{set}
of all comparable pairs, $X_1=\{ (x,x')\in X \mid x\leq x'\}$. Let 
$d^0, d^1:X_1\to X_0$ be the maps $d^0(x,x')=x$, $d^1(x,x')=x'$, with common right inverse $i:X_0\to X_1$, $i(x)=(x,x)$. Then the obvious (monotone) map $c:DX_0\to X$, $c(x)=x$, exhibits $X$ as the coinserter in $\Pos$ of the reflexive pair of discrete posets (also called the \emph{truncated nerve} of the poset)
\begin{equation}
\label{eq:poset_coins}\tag{$N_X$}
\xymatrix@C=35pt{DX_1 \ar@{<-}`d[r]`r[r]_{Di} \ar@<0.5ex>[r]^{Dd^0} \ar@<-0.5ex>[r]_{Dd^1} & DX_0 \ar[r]^-c & X}
\end{equation} 
\end{proposition}

\begin{proof}
We leave to the reader to check the straightforward details.  
\end{proof}

\medskip

\begin{definition}[\textbf{Kan extension}]
\label{def:Lan}
Let $J:\mathscr A\to \K$, 
$H:\mathscr A \to \mathscr L$ be locally monotone functors. 
A $\Pos$-enriched \emph{left Kan extension of $H$ along $J$}, is a locally monotone functor $\Lan_J H:\K\to \mathscr L$, 
such that there is a $\Pos$-natural isomorphism 
\[
(\Lan_J H) X \cong \K(J-, X) * H
\] 
for each $X$ in $\K$.
\end{definition}

\begin{remark}
\label{rem:lan}
\mbox{}\hfill
\begin{enumerate}
\item 
\label{rem:weak_def_lan}
For any locally monotone functor 
$H':\K\to \mathscr L$, there is an isomorphism
\begin{equation}\label{iso_lan}
[\K, \mathscr L](\Lan_J H, H')
\cong 
[\mathscr A, \mathscr L](H, H' J)
\end{equation}
in analogy to the case of unenriched left Kan extensions. 
In particular, there is a $\Pos$-natural transformation 
$\alpha:H\to (\Lan_J H ) J$, 
called {\em unit\/} of the left Kan extension,
which is {\em universal\/} 
in the sense that for a locally monotone functor 
$H':\K \to \mathscr L$, any $\Pos$-natural 
transformation $H\to H' J$ factorises through 
$\alpha$. 

\medskip\noindent In the general enriched setting, requiring 
the isomorphism~\eqref{iso_lan} 
is \emph{strictly weaker} than Definition~\ref{def:Lan}, 
but if $\mathscr L$ is 
\emph{powered}~(also called \emph{cotensored},
see~\cite[Section~3.7]{kelly:enriched}), 
it can however be taken as an alternative definition of 
left Kan extensions (see the discussion after Equation~(4.45) 
in~\cite{kelly:enriched}).
\item 
Suppose $J:\mathscr A \to \K$ is fully faithful. 
Then the unit $\alpha:H \to ( \Lan_J H )J$  
of the left Kan extension is an 
isomorphism~\cite[Proposition~4.23]{kelly:enriched}. 
\item 
By general enriched category theory, the $\Pos$-enriched 
left Kan extension $\Lan_J H$ exists whenever 
$\mathscr A$  is small and $\mathscr L$ is cocomplete. 
But it might exist even when $\mathscr A$ is not small, 
as we shall see later in a special case
(Theorem~\ref{thm:posetification}).
\item 
From~Equation~\eqref{iso_lan} it follows that any locally 
monotone left adjoint 
$Q : \mathscr L \to \mathscr L'$ preserves 
the $\Pos$-enriched left Kan extension $\Lan_J H$, 
in the sense that the canonical natural transformation
$$
\Lan_J (QH) \to Q\ \Lan_J H
$$
is an isomorphism.
\end{enumerate}
\end{remark}

\begin{example}
Recall that we have denoted by $D\dashv V:\Pos\to\Set$ 
the (ordinary, i.e., not locally monotone) 
adjunction between the  discrete and the forgetful functors, 
and by $\mathcal P:\Set\to\Set$ the powerset functor. Then 
the $\Pos$-enriched left Kan extension of $D\mathcal P$ along $D$ 
is the convex powerset functor (see~\cite{kurz+velebil:presentations}, but also Example~\ref{ex:posetif_powerset}). On the other hand, the {\em ordinary\/} left Kan extension of 
$D\mathcal P$ along $D$ is the composite $D\mathcal PV$, 
which is less interesting, as it maps any poset to 
the discrete poset of its subsets. 
\end{example}

%========================================================%

\subsection{Ordered varieties}

We have seen in Section~\ref{sec:coalgebras} a close interplay between modal logic and varieties of algebras. The theory of (locally monotone) $\Pos$-functors and their logics of monotone modal operators naturally leads to the world of ordered varieties, as defined by Bloom and Wright in~\cite{bloom+wright}.

\medskip\noindent More precisely, recall that a (finitary) {\em signature\/} 
$\Sigma$ associates to each natural number $n$ 
a \emph{set} of operation symbols $\Sigma_n$ of arity $n$. 
A {\em $\Sigma$-algebra\/} consists of a \emph{poset} 
$A$ and, for each $\sigma\in\Sigma_n$, a \emph{monotone} 
operation $\sigma_A: A^n\to A$. An ordered variety is specified 
by a signature $\Sigma$ and a set of inequations. 
Bloom~\cite{bloom} proved that varieties are precisely the 
HSP-closed subclasses of algebras for a signature, 
provided that we understand closure under H as 
the closure under surjective homomorphisms and 
the closure under S as closure under {\em embeddings\/} 
(injective and order-reflecting homomorphisms). 

\medskip\noindent The structure theory of ordered varieties is similar to 
the one for ordinary varieties. For more details we refer 
the reader to the original~\cite{bloom+wright} and to the 
more recent paper~\cite{kurz+velebil:varieties}.

\begin{example}
The category $\DL$ of distributive lattices is a variety over $\Pos$ if we take algebras to be ordered in the lattice order: 
$$
a\le b \Leftrightarrow a\wedge b=a
$$
The category $\BA$ of Boolean algebras is a variety over $\Pos$ if we take Boolean algebras to be discretely ordered.
\end{example}

\medskip\noindent Notice that Boolean algebras can \emph{only} be discretely ordered, because of the requirement that operations of ordered algebras should be monotone. In the case of Boolean algebras it is not hard to show that the discrete order is the only one that makes all operations (including negation) monotone (see Section~\ref{discreteBA} below). 

%========================================================%

\subsection{Sifted weights and sifted colimits; strongly 
finitary functors}\label{sec:sifted} 

There is a well-known result that a finitary $\Set$-endofunctor also preserves sifted colimits, or equivalently, filtered colimits and reflexive coequalizers~\cite[Corollary~6.30]{ARV}. Below we sketch the corresponding $\Pos$-enriched theory 
(for more details, we refer to~\cite{bourke,kelly-lack,kurz+velebil:varieties,lack-rosicky:lawvere}). 
A weight $\varphi:\mathscr D^\op\to \Pos$ is called 
\emph{sifted} if finite products commute with $\varphi$-colimits 
in $\Pos$~\cite{kelly-lack}. 
Equivalently, if the 2-functor 
$\varphi*-:[\mathscr D, \Pos]\to \Pos$ preserves $\Pos$-enriched finite products. 
A \emph{sifted colimit} is a colimit weighted by a sifted weight.
Examples of sifted colimits are filtered colimits and 
reflexive coequalizers, but also reflexive coinserters
(see~\cite{bourke}). 

\medskip\noindent There is a close interplay between (ordered) varieties and (enriched) sifted colimits, see also Section~\ref{ak}. For now, remember that in the non-enriched setting, a functor on a variety preserves ordinary sifted colimits if and only if it preserves filtered colimits and reflexive 
coequalizers~\cite{kurz-rosicky:strong}. In the $\Pos$-enriched setting, a locally monotone functor on an ordered variety preserves (enriched) sifted colimits if and only if it preserves filtered colimits and reflexive coinserters~\cite[Proposition~6.8]{kurz+velebil:varieties}.

\medskip\noindent Let $\Set_f$ be the category of finite sets and maps, and  let 
$\iota$ denote the inclusion 
$$
\iota\equiv\Set_f\overset{I}{\hookrightarrow} \Set\overset{D}{\to}\Pos
$$ 
In~\cite{kurz+velebil:varieties}, following~\cite[Theorem~8.4]{bourke}, it was noticed that $\Pos$ is the {\em free cocompletion\/} 
of $\Set_f$ under enriched sifted colimits.\footnote{In general,
let $\Phi$ be a class of weights and let $\K$ be any $\Pos$-category. 
Following~\cite{albert-kelly}, let 
$\Phi\mbox{-}\mathsf{Cocts}$ be the 2-category of 
$\Phi$-cocomplete categories, $\Phi$-cocontinuous functors, 
and natural transformations. The free cocompletion of $\K$ 
with respect to $\Phi$, denoted by $\iota: \K\hookrightarrow \Phi(\K)$,
is uniquely characterized by the property that composition with $\iota$
induces an equivalence 
$
\Phi\mbox{-}\mathsf{Cocts}[\Phi(\K), \mathscr L] 
\cong 
[\K, \mathscr L]
$ 
of categories.
Its inverse is given by left Kan extension along $\iota$.} 
Briefly, it means that every poset can be expressed as a canonical
sifted colimit of finite discrete posets. This colimit can be 
`decomposed': every poset is a filtered colimit of finite posets, 
which in turn arise as reflexive coinserters of discrete finite 
(po)sets, as explained in Proposition~\ref{poset=refl_coins_discrete}. 

\begin{definition}(\cite{kelly-lack}) \label{def:strongly_finitary}
A \emph{strongly finitary} functor $T':\Pos\to \Pos$ is 
a locally monotone functor isomorphic to the $\Pos$-enriched 
left Kan extension along $\iota:\Set_f\to\Pos$ 
of its restriction along $\iota$, 
that is, $T'\cong \Lan_{\iota}(T'\iota)$ holds. 
\end{definition}

\noindent Thus, strongly finitary endofunctors of $\Pos$ are 
{\em precisely\/} the locally monotone endofunctors of $\Pos$ 
that preserve (enriched) sifted colimits. 

\medskip\noindent Recall the examples  $\BA$ and $\DL$ of $\Pos$-categories. 
They are connected by the monadic {\em enriched\/} 
adjunctions 
$$
F\dashv U:\BA\to\Set,
\quad 
F'\dashv U':\DL\to\Pos,
$$ 
where $U$ and $U'$ are the corresponding forgetful functors. Extending the notations employed in Section~\ref{sec:Presentations}, we put 
$$\mathbf J:\BA_\ff\to\BA \quad\quad \text{ and } \quad\quad\mathbf J':\DL_\ff\to\DL $$
 to be the inclusion functors of the full subcategories spanned by the algebras which are \emph{free on finite (discrete po)sets}. 

\begin{lemma}\label{dense lemma}
The inclusion functors $\mathbf J$ and $\mathbf J'$ exhibit $\BA$, respectively $\DL$, as the free enriched cocompletions under sifted colimits of $\BA_\ff$ and $\DL_\ff$. In particular, these functors are dense.%
\footnote{A general functor $H: \mathscr A\to \mathscr K$ 
is \emph{dense} if the left Kan extension 
of $H$ along itself is (naturally isomorphic to) 
the identity functor on $\mathscr K$; 
that is, each $X$ of $\K$ can be expressed as a canonical colimit 
$\K(H-,X)* H$~\cite[Chapter~5]{kelly:enriched}.
}
\end{lemma}

\begin{proof}
We know that the functor $\mathbf J:\BA_\ff\to\BA$ exhibits $\BA$ as a free cocompletion under sifted colimits 
(see~\cite{kurz-rosicky:strong}). Now the conclusion for $\mathbf J$ follows because of the discrete enrichment. 

For the inclusion functor $\mathbf J'$ for distributive lattices, the required result is an instance 
of~\cite[~Theorem~6.10]{kurz+velebil:varieties}, since $\DL$ is a finitary variety of
ordered algebras (thus, $\DL$ is isomorphic to the category of
algebras for a strongly finitary monad on $\Pos$).
\end{proof}

\begin{corollary}\label{cor:when L is lan}
A functor $L:\BA\to\BA$ has the form $\Lan_\mathbf J (L\mathbf J)$
if and only if it preserves (ordinary) sifted colimits. 
A functor $L':\DL\to\DL$ has the form $\Lan_{\mathbf J'} (L'\mathbf J')$
if and only if it preserves sifted colimits.  
\end{corollary}

%========================================================%

\section{Presenting functors on $\Pos$}\label{sec:Pos-functors}

\medskip\noindent
For reasons explained in the introduction, we are interested in 
the {\em posetifications\/} of functors $T:\Set\to\Set$. 
Technically, posetifications can be described as enriched left Kan
extensions of the functors $DT:\Set\to\Pos$. This suggests to also
investigate the more general question of when a left Kan extension 
of a functor $H:\Set\to\Pos$ exists. By general arguments, we know 
that such a left Kan extension exists if the functor $H$ is 
finitary, but that would exclude 
the example $T=\mathcal P$ from the introduction. Therefore, 
in Section~\ref{sec:posetification}, we show that, in fact, 
\emph{any} functor $H:\Set\to\Pos$ has an enriched left Kan 
extension. Then, in Section~\ref{sec:posetification2}, we 
characterize posetifications among all functors $\Pos\to\Pos$.

%========================================================%

\subsection{Posetifications and functors $\Pos\to\Pos$ with presentations in discrete arities}\label{sec:posetification}

\medskip\noindent In order to relate endofunctors of $\Set$ and of 
$\Pos$, we give below an improved version 
of~\cite[Definition~1]{calco2011}:

\begin{definition}
\label{def:extension}
Let $T$ be an endofunctor on $\Set$. An endofunctor 
$T':\Pos\to \Pos$ is said to be a \emph{$\Pos$-extension} of $T$ if $T'$ is locally monotone and if the square
\begin{equation}\label{eq:extension}
\vcenter{
\xymatrix{
\Pos
\ar[0,1]^-{T'}
\ar@{<-}[1,0]_{D}
\ar@{}[1,1]|{\nwarrow\alpha}
&
\Pos
\ar@{<-}[1,0]^{D} 
\\
\Set
\ar[0,1]_-{T}
&
\Set
\\
}
}
\end{equation} 
commutes up to a natural isomorphism $\alpha:DT\to T'D$. 

A $\Pos$-extension $ T'$ is called the \emph{posetification} of $T$ if the above square exhibits $T'$ as $\Lan_D(DT)$ (in the $\Pos$-enriched sense), 
having $\alpha$ as its unit.
\end{definition} 

\begin{remark}\label{rem:posetific}
\mbox{}\hfill
\begin{enumerate}
\item
Any extension of $T$ has to coincide with $T$ on discrete sets. 
One would be tempted to take $T'= DTV$ as an extension of $T$; 
but this is not necessarily locally monotone, as $V$ 
fails to be so. There is also the possibility of 
choosing $T'=DTC$, which does produce an extension, but 
not the posetification.
In fact, $DTC$ is the $\Pos$-enriched \emph{right} 
Kan extension $\mathrm{Ran}_D (DT)$. 
\item \label{DC}
Extensions are not necessarily unique. For example, 
the identity functor on $\Pos$ obviously extends the identity 
functor on $\Set$, but the same does the functor $DC$ sending 
a poset to the (discrete) set of its connected components. 
\item 
In general extensions do not need to inherit all the properties 
of the $\Set$-functors that they extend. For example, extensions 
of finitary functors are not necessarily finitary: consider the finitary functor on $\Set$ 
which maps a set $X$ to the set of almost constant sequences on $X$, 
\[
TX=\{l:\mathbbm N \to X \mid l(n)=l(n+1) \text{ for all but a
    finite number of }n\}
    \]
It admits the $\Pos$-extension 
\[
\qquad \quad T'(X,\leq)=
\{l:(\mathbbm N, \leq) \to (X,\leq) \mid 
l(n)\leq l(n+1) 
\text{ for all but a finite number of }n\}
\]
mapping a poset $(X, \leq)$ to the poset of almost monotone 
sequences on $X$, ordered component-wise. But this extension 
$T'$ is not finitary: to see this, consider the family of finite 
posets $(\{0, \ldots, n-1\})_{n<\omega}$ with the usual order, 
with inclusion maps, whose colimit in $\Pos$ is $(\mathbbm N,\leq)$.
Then one can easily check that $T'$ does not preserve the above 
filtered colimit. 
\end{enumerate}
\end{remark}

%========================================================%

\medskip\noindent 
It is clear from general considerations that every \emph{finitary} 
endofunctor of $\Set$ has a posetification. The point 
of the next theorem is to drop the finitarity restriction.

\begin{theorem}\label{thm:posetification}
Every endofunctor of $\Set$ has a posetification. 
\end{theorem}

\begin{proof}
The posetification of a functor $T:\Set\to\Set$ is constructed as follows. Recall from Proposition~\ref{poset=refl_coins_discrete} that each poset $X$ can be expressed as a reflexive coinserter in $\Pos$, as follows:
\[
\xymatrix{DX_1 \ar@{<-}`d[r]`r[r]_{Di_X} \ar@<0.35ex>[r]^{Dd_X^0} \ar@<-0.35ex>[r]_{Dd_X^1}& DX_0 \ar[r]^-{c_X} & X}
\]
Notice that unlike Proposition~\ref{poset=refl_coins_discrete}, in the above we have added subscripts to emphasize the poset $X$. 

Denote by $e_X:DTX_0\to {T'}X$ the coinserter in $\Pos$ of
the (reflexive) pair $(DTd_X^0, DTd_X^1)$:
\begin{equation}\label{eq:def-T'}
\xymatrix@C=35pt{DTX_1 \ar@{<-}`d[r]`r[r]_{DTi} \ar@<0.35ex>[r]^{DTd^0_X} \ar@<-0.35ex>[r]_{DTd^1_X}& DTX_0 \ar[r]^-{e_X} & T'X}
\end{equation}
We claim that the assignment $X\mapsto{T'}X$ extends to a locally monotone functor
${T'}:\Pos\to\Pos$, and that ${T'}\cong\Lan_{D}{(DT)}$ holds.  

\medskip\noindent 
We proceed in several steps.
\begin{enumerate}
\item 
Consider a monotone map $f:X\to Y$. It induces the obvious
maps $f_0:X_0\to Y_0$ and $f_1:X_1\to Y_1$. Moreover, the squares
$$
\xymatrix{
X_1
\ar[0,1]^-{f_1}
\ar[1,0]_{d^0_X}
&
Y_1
\ar[1,0]^{d^0_Y}
&
&
X_1
\ar[0,1]^-{f_1}
\ar[1,0]_{d^1_X}
&
Y_1
\ar[1,0]^{d^1_Y}
\\
X_0
\ar[0,1]_-{f_0}
&
Y_0
&
&
X_0
\ar[0,1]_-{f_0}
&
Y_0
}
$$
commute. Thus, we have the inequality 
$$
e_Y\circ DTf_0\circ DTd^0_X \leq e_Y\circ DTf_0\circ DTd^1_X
$$
since 
\[
DTf_0\circ DTd^0_X=DTd^0_Y\circ DTf_1 \qquad DTf_0\circ DTd^1_X=DTd^1_Y\circ DTf_1
\]
and
\[
e_Y\circ DTd^0_Y\leq e_Y\circ DTd^1_Y
\]
hold.

Hence one can define ${T'}f:{T'}X\to{T'}Y$ as the unique
mediating monotone map (using the co-universality of the coinserter $e_X$).
\item
The 1-dimensional aspect of coinserters proves immediately
that ${T'}$ preserves composition and identity; that is, $T'$ is an ordinary functor $\Pos\to \Pos$. 
\item
We show that $T'$ is locally monotone; that is, ${T'}f\leq{T'}g$ whenever $f\leq g$ holds, for monotone maps $f,g:X\to Y$. Observe
that $f\leq g$ yields a map $\tau:X_0\to Y_1$, $x\mapsto (f(x),g(x))$, such that the triangles
$$
\xymatrix{
X_0
\ar[0,1]^-{\tau}
\ar[1,1]_{f_0}
&
Y_1
\ar[1,0]^{d^0_Y}
&
&
X_0
\ar[0,1]^-{\tau}
\ar[1,1]_{g_0}
&
Y_1
\ar[1,0]^{d^1_Y}
\\
&
Y_0
&
&
&
Y_0
}
$$
commute.

To prove ${T'}f\leq{T'}g$, it is enough to check that 
${T'}f\circ e_X\leq {T'}g\circ e_X$ holds, for we can then use
the 2-dimensional aspect of coinserter $e_X$. This inequality follows from
\begin{align*}
{T'}f\circ e_X = 
e_Y \circ DTf_0 = 
e_Y \circ DTd^0_Y \circ DT\tau \\
{T'}g\circ e_X = 
e_Y \circ DTg_0 = 
e_Y \circ DTd^1_Y \circ DT\tau,
\end{align*}
and from the fact that $e_Y\circ DTd^0_Y \leq e_Y\circ DTd^1_Y$ holds.
\item
To prove ${T'}\cong\Lan_{D}{(DT)}$, we shall show that there
is an isomorphism between the poset of natural transformations ${T'}\to H$
and the poset of natural transformations $DT\to HD$, for every locally monotone $H:\Pos\to\Pos$ (see Remark~\ref{rem:lan}\eqref{rem:weak_def_lan}). 
\begin{enumerate}
\item 
Consider a natural transformation $\alpha:DT\to HD$.
For every poset $X$, we define $\check{\alpha}_X:{T'}X\to HX$
as the unique mediating map out of a coinserter:
$$
\xymatrix{
DTX_0
\ar[0,1]^-{e_X}
\ar[1,0]_{\alpha_{X_0}}
&
{T'}X
\ar@{.>}[1,0]^{\check{\alpha}_X}
\\
HDX_0
\ar[0,1]_-{Hc_X}
&
HX
}
$$
Recall that, above, $c_X:DX_0\to X$ is a coinserter of $Dd^0_X$, $Dd^1_X$.

The above definition makes sense since 
$$
Hc_X\circ\alpha_{X_0}\circ DTd^0_X \leq Hc_X\circ\alpha_{X_0}\circ DTd^1_X
$$
holds: the equalities 
$$
\alpha_{X_0}\circ DTd^0_X=HDd^0_X\circ\alpha_{X_1} 
\quad \mbox{and} \quad 
\alpha_{X_0}\circ DTd^1_X=HDd^1_X\circ\alpha_{X_1}
$$ 
follow by naturality, and 
$$
c_X\circ d^0_X\leq c_X\circ d^1_X
$$
holds, since $c_X$ is a coinserter.

We prove that $\check{\alpha}$ is natural. Consider any monotone map $f:X\to Y$ and compare 
\[
\qquad \qquad \vcenter{
\xymatrix{
DTX_0
\ar[0,1]^-{e_X}
\ar[1,0]_{DTf_0}
&
{T'}X
\ar[1,0]^{{T'}f}
\\
DTY_0
\ar[0,1]^-{e_Y}
\ar[1,0]_{\alpha_{Y_0}}
&
{T'}Y
\ar[1,0]^{\check{\alpha}_Y}
\\
HDY_0
\ar[0,1]_-{Hc_Y}
&
HY
}
}
\mbox{\qquad with \qquad}
\vcenter{
\xymatrix{
DTX_0
\ar[0,1]^-{e_X}
\ar[1,0]_{\alpha_{X_0}}
&
{T'}X
\ar[1,0]^{\check{\alpha}_X}
\\
HDX_0
\ar[0,1]^-{Hc_X}
\ar[1,0]_{HDf_0}
&
HX
\ar[1,0]^{Hf}
\\
HDY_0
\ar[0,1]_-{Hc_Y}
&
HY
}}
\]
Using naturality of $\alpha$ and co-universality of 
$e_X$, we conclude $Hf\circ\check{\alpha}_X=\check{\alpha}_Y\circ{T'}f$.
\item
Given a natural transformation $\beta:{T'}\to H$, we define, for every
set $X_0$, the mapping $\wh{\beta}_{X_0}:DTX_0\to HDX_0$ to be 
$\beta_{DX_0}:{T'}DX_0\to HDX_0$ (Here we have used the fact that ${T'}DX_0$ is naturally isomorphic to $DTX_0$).
\item
It is easy then to see that the assignments $\alpha\mapsto\check{\alpha}$ and $\beta\mapsto\wh{\beta}$
are monotone and inverse to each other.\qedhere
\end{enumerate}
\end{enumerate}
\end{proof}

\medskip\noindent As a corollary of the proof of 
the above theorem (replace $DT$ by $H$ everywhere in the 
above proof) we obtain

\begin{corollary}
For every functor $H:\Set\to\Pos$, the $\Pos$-enriched 
left Kan extension $\Lan_{D}{H}:\Pos\to\Pos$ exists.
\end{corollary}

\medskip\noindent The following gives an example where the construction of posetification as given by Equation~\eqref{eq:def-T'} can be obtained straightforwardly.

\begin{example}\label{ex:posetif_powerset}
Let $T:\Set\to\Set$ be the covariant powerset functor $\pcal$ and let $X$ be a poset. We shall see how to use Equation~\eqref{eq:def-T'} to determine $\pcal' X$. Taking into account that $D\pcal d^0_X$ and $D\pcal d^1_X$ are the direct images of the projections, the coinserter \eqref{eq:def-T'} becomes
\[
\xymatrix@C=35pt{D\pcal X_1 
\ar@<0.35ex>[r]^{D\pcal d^0_X} \ar@<-0.35ex>[r]_{D\pcal d^1_X}
& D\pcal X_0 \ar[r]^-{e_X} & \pcal 'X}
\]
Recall from Remark~\ref{rem:construction_coinserter} how coinserters are built in $\Pos$: first, consider the relation $\mathbf r$ on $\mathcal PX_0$ given by:
\begin{equation}\label{egli_milner_coinserter}
Y \mathbf r Z \quad \Leftrightarrow \quad \exists S\subseteq X_1 \ . \ \mathcal P d^0_X (S)=Y \quad \mbox{ and } \quad \mathcal P d^1_X(S)=Z
\end{equation}
for subsets $Y,Z\subseteq X_0$ (in fact, in
Remark~\ref{rem:construction_coinserter} we have considered the
transitive closure of relation~\eqref{egli_milner_coinserter}). Unravelling the above, we obtain that $Y \mathbf r Z $ if and only if 
\begin{equation}\label{eq:eglimilner}
\forall y\in Y \,.\, \exists z\in Z \,.\, y\le z  \quad \mbox{ and } \quad 
\forall z\in Z \,.\, \exists y\in Y \,.\, y\le z  
\end{equation}
The relation $\mathbf r$ determined by~Equation~\eqref{eq:eglimilner} is known as the Egli-Milner order: it is reflexive and transitive (consequently, it coincides with its transitive closure), and two subsets $Y,Z$ of $X_0$ are equivalent with respect to $\mathbf r$ if and only if they have the same convex closure. Thus $\pcal 'X$, obtained by quotienting $\mathcal PX_0$ with respect to the equivalence relation induced by $\mathbf r$, is the set of convex subsets of $X$ ordered by the Egli-Milner order.

Similarly, if we let $H:\Set\to\Pos$ with $HX=(\pcal X,\subseteq)$, we obtain the downset functor $\mathcal D:\Pos\to\Pos$ with $\mathcal DX$ ordered by inclusion as the left Kan extension of $H$ along $D$. Dually, if $H:\Set\to\Pos$ is given by $HX=(\pcal X,\supseteq)$, we have that $\Lan_D H$ is the upset functor $\mathcal U:\Pos\to\Pos$ with $\mathcal UX$ ordered by reverse inclusion. To verify this, note that two subsets are equivalent according to the left-hand side of \eqref{eq:eglimilner} if and only if they have the same downset closure and two subsets are equivalent according to the right-hand side  if and only if they have the same upset closure.
\end{example}

\begin{remark}[\textbf{On presentations by monotone operations and equations \linebreak in discrete arities}] 
\label{RemPos}
\mbox{}\hfill
\begin{enumerate}
\item %\label{finPos}
The posetification built in 
Theorem~\ref{thm:posetification} coincides with the one from~\cite[(3.2)]{calco2011} given by the coequalizer in $\Pos$
\begin{equation}\label{eq:coequ-posetification}
\entrymodifiers={!!<0pt,.7ex>+}
\xymatrix@=10pt@1{
\underset{m,n<\omega}{\coprod} 
\mathsf{Set}(m,n) \times Tm \times [D n,X] 
\ar@<-0.5ex>[r]%_-{\lambda} 
\ar@<0.5ex>[r]%^-{\rho} 
& 
\underset{n<\omega}{\coprod}
Tn
\times 
[Dn, X] 
\ar[r]%^-\pi 
&
\Lan_D(DT)(X)
} 
\end{equation}
if $T$ is finitary (this follows from the fact that $\Lan_{D}{(DT)}\cong
\Lan_{DI}(DTI)$, where remember from Section~\ref{sec:sifted} that $I:\Set_f\to\Set$ denotes the inclusion functor). 

\item\label{RemPos:item-present1}
Let us explain how~Equation~\eqref{eq:coequ-posetification} gives \emph{a presentation by monotone operations and equations in discrete arities}. The operations of arity $n$ are given by $Tn$. They are necessarily monotone because the arguments $[Dn,X]$ form a poset and we take the coequalizer in $\Pos$. The arities are discrete because $m,n$ range over sets, not posets. For each pair $(m,n)$, we have a poset of equations $\mathsf{Set}(m,n) \times Tm \times [D n,X]$ (where the order on the equations does not play a role in the computation of the coequalizer).

\item 
For an explicit example of such a presentation by operations and equations, consider $T$ to be the finite powerset functor. First, recall that it can be presented in $\Set$ as the quotient of $\coprod_{n<\omega} X^n$ by a set of equations 
specifying that the order and the multiplicity in which elements of the set $X$ occur in lists in $X^n$ does not matter. Second, with $X$ now standing for a poset, note that according 
to~\cite[Proposition~5]{calco2011}, we obtain the posetification of $T$ by quotienting $\coprod_{n<\omega} [Dn, X]$ in $\Pos$ by the same equations. It is not difficult to show that this gives us the (finite) convex powerset functor on 
$\Pos$~\cite[Proposition~5.1]{kurz+velebil:presentations}.

\item If we generalize from the posetification of a finitary functor $\Set\to\Set$ to the left Kan extension of a finitary functor $\Set\to\Pos$ the formula~\eqref{eq:coequ-posetification} is still available and we obtain the same presentations as in item (\ref{RemPos:item-present1}), just that the $Tn$ need not be discrete anymore. For example, if we let the $Tn$ in~Equation~\eqref{eq:coequ-posetification} be $\mathcal P(n)$ ordered by inclusion, we get a presentation of the functor $\Pos\to\Pos$ mapping a poset $X$ to the set of finitely generated downsets ordered by inclusion (Hoare powerdomain).

\item If we generalize further, giving up that the functor be
  finitary, we lose the formula~\eqref{eq:coequ-posetification} since
  the large coproducts may not exist in $\Pos$. Nevertheless, we can still interpret a functor $T:\Set\to\Pos$ as having \emph{a presentation by monotone operations and equations in discrete arities}. This time the arities range over all cardinals, so that for each cardinal $\aleph$ we have a poset of operations $T(\aleph)$ and for each pair $(\aleph,\aleph')$ 
of cardinals  we have a set of equations 
$\mathsf{Set}(\aleph,\aleph') \times T(\aleph) \times [D(\aleph'),X]$.
\end{enumerate}
\end{remark}

\noindent We summarize this discussion by making the following definition.

\begin{definition}\label{def:discretearities}
We say that a  functor $T':\Pos\to\Pos$  has a \emph{presentation in discrete arities} if $T'=\Lan_D H$ for some functor $H:\Set\to\Pos$.
\end{definition}

\medskip\noindent The posetification of a functor $T:\Set\to\Set$ has been defined by quotienting with respect to the relation given by applying $T$ to the nerve of a poset. Not surprisingly, this construction is closely related to the notion of relation lifting of the functor $T$:

%========================================================%

\begin{remark}[\textbf{On posetifications and relation lifting}]\label{wpb-exsq} 
\mbox{}\hfill
\begin{enumerate}
\item \label{RelLift}
Let $T:\Set\to \Set$ be an arbitrary functor. For a relation 
$\mathbf r\subseteq X\times Y$, recall that the $T$-relation 
lifting of $\mathbf r$ is given by the epi-mono factorisation
as on the right of the following diagram (see for 
example~\cite{barr,ckw,thijs, trnkova}): 
\[\xymatrix@C=35pt@R=19pt{
&\mathbf r \ar[ddl]_{\pi_1} \ar[ddr]^{\pi_2} \ar@{>->}[dd]
\\ & {} \ar@{}[d] &
\\
X & X\times Y \ar[l] \ar[r] & Y 
}
\qquad 
\xymatrix@C=35pt@R=13pt{
&T\mathbf r \ar[ddl]_{T\pi_1}\ar[ddr]^{T\pi_2} \ar@{->>}[d] 
\\ & \mathbf{Rel}_T(\mathbf r) \ar@{>->}[d]\\
TX & TX\times TY \ar[l] \ar[r] & TY 
}
\]
Explicitly, 
\[
\mathbf{Rel}_T(\mathbf r)= \{ (u,v)\in TX\times TY \mid \exists w\in T\mathbf r \ . \ T\pi_1 (w)=u \ \land \ T\pi_2 (w)=v \}
\]
The relation lifting satisfies the following properties:
\begin{enumerate}

\item \label{diag}

It preserves the equality relation: $=_{TX}\ =\ \mathbf{Rel}_T(=_X)$.

\item 

It preserves the inclusion of relations: if $\mathbf r\subseteq \mathbf s$,  then 
$\mathbf{Rel}_T(\mathbf r)\subseteq \mathbf{Rel}_T(\mathbf s)$.

\item \label{composRelLift}%Rel_T is colax

If $\mathbf r\subseteq X\times Y$ and $\mathbf s\subseteq Y\times
Z$, then 
$$
\mathbf{Rel}_T(\mathbf s\circ \mathbf r)\subseteq \mathbf{Rel}_T(\mathbf s)\circ \mathbf{Rel}_T(\mathbf r)
$$ 
with equality if and only if $T$ preserves weak pullbacks.

\item 

It preserves converses of relations: $\mathbf{Rel}_T(\mathbf r^\op)=\mathbf{Rel}_T(\mathbf r)^\op$.

\item 

Given functions $f:X\rightarrow X'$, $g:Y\rightarrow Y'$ and relation $\mathbf r'\subseteq X'\times Y'$, then 
\[ \mathbf{Rel}_T((f\times g)^{-1}(\mathbf r')) \subseteq (Tf\times Tg)^{-1}(\mathbf{Rel}_T(\mathbf r')) \]
with equality if $T$ preserves weak pullbacks.
\end{enumerate} 
\item In addition to the above, we should also mention the (less-known?) fact that relation lifting commutes with functor composition, in the sense that 
\[ 
\mathbf{Rel}_{TS}(\mathbf{r})=\mathbf{Rel}_T (\mathbf{Rel}_S (\mathbf{r}))
\]
for any relation $\mathbf r\subseteq X\times Y$ and any 
endofunctors $T$, $S$ of $\Set$
(see~\cite[Section~4.4]{ckw}, and use that, 
assuming the axiom of choice, any endofunctor of
$\Set$ preserves (strong) epimorphisms, 
i.e., surjective maps).

\item \label{posetif=rel-lift}
Recall again that the posetification $T'$ of a $\Set$-endofunctor $T$ was obtained via coinserters, 
\[
\xymatrix@C=30pt{DTX_1 \ar@<0.35ex>[r]^{DTd^0} \ar@<-0.35ex>[r]_{DTd^1}& DTX_0 \ar[r]^e & T'X}
\]
for any poset $X$. Observe in fact that the relation $\mathbf{r}$ described in Remark~\ref{rem:construction_coinserter} at the first stage of the coinserter construction, for the pair of (monotone) maps $DTd^0$ and $DTd^1$, is precisely the transitive closure of the $T$-relation lifting $\mathbf{Rel}_T(X_1)$ of the order $X_1$ on $X$. By~\hyperref[diag]{(\ref*{RelLift})(a)} above, $\mathbf{Rel}_T(X_1)$ is reflexive, and by~\hyperref[composRelLift]{(\ref*{RelLift})(c)} it is also transitive if $T$ preserves weak pullbacks. If this is the case, then the posetification $T'$ can be explicitly described as mapping a poset $X$ to the quotient poset of the preordered set $(TX_0, \mathbf{Rel}_T(X_1))$.

\item 
The above two items provide a proof that, for $T$, $S$, endofunctors of $\Set$ preserving weak pullbacks, the isomorphism
$$
(TS)'
\cong
T'S'
$$
holds for their posetifications.
\end{enumerate}
\end{remark}

\medskip\noindent The property of a functor $T:\Set\to\Set$ preserving weak pullbacks plays an important role in the theory of coalgebras \cite{rutten:uc}. In the category $\Pos$, the following concept is the enriched
analogue of a weak pullback.

\begin{definition}[\cite{Guitart80}]\label{def:exact}
An \emph{exact square} in the category $\Pos$ of posets, or in the category $\Pre$ of preorders, is a diagram
\begin{equation}
\label{exact-sq}
\vcenter{
\xymatrix{
E \ar[r]^{\alpha} \ar[d]_{\beta} \ar@{}[1,1]|{\swarrow}& X\ar[d]^f \\
Y\ar[r]_g & Z}
}
\end{equation}
with $f\circ \alpha\leq g\circ \beta$, such that
\begin{equation*}\label{exact}
\forall x\in X,y\in Y.\ f(x)\leq g(y)\ \Rightarrow \exists \ w\in E.\ (x\leq \alpha(w) \land \beta(w)\leq y)
\end{equation*}

\end{definition}

\medskip\noindent
That the above concept generalises weak pullbacks in sets
is seen as follows.
An exact square of {\em discrete\/} posets is precisely 
a weak pullback of their underlying sets. Equivalently, 
the discrete functor $D:\Set\to\Pos$ 
maps weak pullbacks to exact squares 
and reflects exact squares to weak pullbacks.

\medskip\noindent 
Given an endofunctor $T$ of $\Set$, we shall now connect 
the property of preserving weak pullbacks with the preservation 
of exact squares by the corresponding  posetification $T'$. 
This fact will be used in Theorem~\ref{mainthm} below.

\begin{theorem}\label{thm:wpb_exsq}
Let $T$ be any %finitary 
endofunctor of $\Set$ and let $T'$ be its posetification. 
Then $T$ preserves weak pullbacks if and only if $ T'$ preserves 
exact squares.
\end{theorem}
\begin{proof}
This was proved in~\cite{calco2011} under the additional 
assumption that $T$ is finitary. Here, we present 
an argument valid for all $\Set$-functors. 

\medskip\noindent We start with the easy implication. Assume $T'$ preserves exact squares and consider a weak pullback in $\Set$
\begin{equation}
\label{weakpullback}
\vcenter{
\xymatrix{
E \ar[r]^{\alpha} \ar[d]_{\beta} %\ar@{}[1,1]|{=}
& X\ar[d]^f \\
Y\ar[r]_g & Z
}
}
\end{equation}
Then~Equation~\eqref{weakpullback} is mapped by $D$ to an exact square in $\Pos$, and $T'$ preserves such by hypothesis. Using the isomorphism $DT\cong T'D$, we conclude that 
\[
\xymatrix{
DT E \ar[r]^{DT\alpha} \ar[d]_{DT\beta} %\ar@{}[1,1]|{=}
& DTX\ar[d]^{DTf} \\
DTY\ar[r]_{DTg} & Z
}
\]
is an exact square of discrete posets, that is, a weak pullback in $\Set$. 

\bigskip\noindent Now, we assume that $T$ preserves weak pullbacks 
and we show that its posetification $T'$ preserves exact 
squares. Recall from Remark~\ref{wpb-exsq}\eqref{posetif=rel-lift}
that $T'X$ is the quotient of the preordered set 
$(TX_0, \mathsf{Rel}_T(X_1))$ since $T$ is assumed 
to preserve weak pullbacks.

\medskip\noindent In fact, it is easy to see that for each 
preordered set (poset) $X$, the construct 
$$
\overline{T}X=(TX_0, \mathsf{Rel}_T(X_1))
$$
yields a locally monotone functor $\overline{T}$ 
on the category 
$\mathsf{Preord}$ of preordered sets and monotone mappings.

\medskip\noindent The inclusion functor $\mathsf{Incl}: \Pos\to \mathsf{Preord}$ and 
its left adjoint, the quotient functor 
$\mathsf{Quot}:\mathsf{Preord}\to \Pos$ both preserve 
exact squares~\cite[Example~6.2]{rel-lift}, 
and the composite 
\[
\xymatrixcolsep{2.6pc}
\xymatrix{
\Pos
\ar[0,1]^-{\mathsf{Incl}}
&
\mathsf{Preord}
\ar[0,1]^-{\overline{T}}
&
\mathsf{Preord}
\ar[0,1]^-{\mathsf{Quot}}
&
\Pos
}
\]
is precisely $T'$. Consequently, it is enough to show 
that $\overline{T}$ preserves exact squares. 

\medskip\noindent Consider thus an exact square in $\Pre$:
\begin{equation}\label{eq:exactsquare}
\xymatrix{ E \ar[r]^{\alpha} \ar[d]_{\beta} \ar@{}[1,1]|{\swarrow}& X\ar[d]^f \\ Y\ar[r]_g & Z }
\end{equation}
and follow the steps below: 

\begin{enumerate}
\item First, 
the inequality 
$
\overline{T}(f) \circ \overline{T}(\alpha) 
\leq 
\overline{T}(g)\circ \overline{T}(\beta)$ 
holds since $\overline{T}$ is locally monotone.
\item 
Next, we form the three pullbacks in $\Set$ of the first diagram below, which by hypothesis will be mapped by $T$ to weak 
pullbacks. %To avoid overloaded notation, we shall slightly abuse and we denote by same symbol both the preordered set and its underlying set, and do the same for (monotone) mappings between them. 
As in Proposition~\ref{poset=refl_coins_discrete}, 
here $d^0, d^1:Z_1 \to Z$ stand for the projections from the set 
of comparable pairs to (the underlying set of) $Z$. 
\begin{equation}\label{3pb}
\vcenter{ \xymatrix@C=30pt@R=10pt{
& R  \ar[dl]_{r_0} \ar[dr]^{r_1} & \\
P \ar[dr]^{p_1} \ar[dd]_{p_0} & & Q \ar[dl]_{q_0} \ar[dd]^{q_1} \\
& Z_1 \ar@<-0.5ex>[dd]_{d^0} \ar@<0.5ex>[dd]^{d^1} & \\
X \ar[dr]_f  && Y \ar[dl]^g\\ 
& Z &}}
\quad 
\mapsto 
\quad
\vcenter{
\xymatrix@C=30pt@R=10pt{
& TR  \ar[dl]_{Tr_0} \ar[dr]^{Tr_1} & \\
TP \ar[dr]^{Tp_1}  \ar[dd]_{Tp_0} & & TQ \ar[dl]_{Tq_0} \ar[dd]^{Tq_1} \\
& TZ_1 \ar@<-0.5ex>[dd]_{Td^0} \ar@<0.5ex>[dd]^{Td^1} & \\
X \ar[dr]_{Tf}  && TY \ar[dl]^{Tg}\\ 
& Z &}}
\end{equation}
Explicitly, 
\begin{eqnarray*}
P \qquad 
&=&
\{(x,z)\in X\times Z \mid f(x)\leq z\}
\\
Q \qquad 
&=&
\{(z, y)\in Z\times Y \mid z\leq g(y)\}
\\
R \qquad 
&=&
\{(x,y)\in X \times Y \mid f(x)\leq g(y)  \} 
\\
r_0(x,y)
&=&
(x,g(y))
\\
r_1(x,y)
&=&
(f(x),y)
\end{eqnarray*}

\item From the description of $R$ above, notice that $R$ is non-empty (as we started from an exact square), and that given $(x,y) \in R$, there is some $w \in E$ such that $x\leq \alpha(w)$ and $\beta(w)\leq y$. Assuming the axiom of choice, fix such a $w\in E$ for each $(x,y)\in R$ and define a map $\theta:R\to E$ by $\theta(x,y)=w$. It can be considered monotone if $R$ is taken to be a discrete poset. 

\item Consider the cube below in $\Pre$, where $R,P,Q$ carry the discrete (pre)order.
\[\xymatrix@C=40pt{
R \ar[dd]_{r_1} \ar[rr]^{r_0} & & P \ar'[d][dd]_{p_1} \\ & E %\ar@{}[dl]|{\sswarrow} 
\ar@{<..}[ul]_{\theta} \ar[rr]^(.35){\alpha} \ar[dd]_(.35){\beta} & & X 
%\ar@{}[ddll]|(.35){\swarrow} 
\ar@{<-}[ul]_{p_0} \ar[dd]_(.35){f} \\ Q \ar'[r][rr]^(.48){q_0} \ar[dr]_{q_1} && Z_1 \ar@<0.35ex>[dr]^{d^0} \ar@<-0.35ex>[dr]_{d^1}  \\& Y \ar[rr]^{g} & & Z }
\]
The back, right-hand and bottom faces commute by~Equation~\eqref{3pb}. 
The front face is the exact square of~Equation~\eqref{eq:exactsquare}, 
in particular $f\circ \alpha \leq g\circ \beta$ holds. The 
remaining top and left-hand faces commute laxly, 
in the sense that following inequalities hold: 
\begin{equation}\label{cube}
p_0 \circ r_0 \leq \alpha \circ \theta \quad \mbox{ and } 
\quad 
\beta \circ \theta \leq q_1 \circ r_1
\end{equation}

\item 
We are now able to show that $\overline{T}$ 
preserves exact squares. To that end, let 
$u\in \overline{T}(X)$, $v\in \overline{T}(Y)$ 
such that 
$$
\overline{T}(f)(u)\leq \overline{T}(g)(v)
$$
in $\overline{T}(Z)=(TZ, \mathsf{Rel}_T(Z_1))$. 
That is, $u\in TX$, $v\in TY$ and there exists some 
$w\in T(Z_1)$ such that $Td^0 (w)=Tf(u)$ and  $Td^1(w)=Tg(v)$. 

As all the squares in the second diagram in~Equation~\eqref{3pb} 
are weak pullbacks, we can conclude that there is some 
$\bar w\in TR$ which is mapped to $u\in TX$, respectively $v\in TY$. 
Let $\omega =T\theta(\bar w)\in TE$. Then one can easily check 
using~Equation~\eqref{cube} that $u\leq T\alpha(\omega)$ and 
$T\beta(\omega)\leq v$ hold. 

All in one, we have showed that $\overline{T}$ 
maps an exact square to an exact square. Thus also 
the posetification $T'$ of $T$ preserves exact squares. 
\qedhere 
\end{enumerate}
\end{proof}

\begin{example}
\mbox{}\hfill
\begin{enumerate}
\item Let $T=\Id$ on $\Set$. Then its posetification 
is the identity functor on posets (recall that the discrete-poset functor $D$ is dense, see the last paragraph in Section~2 of~\cite{Street}).
\item If we take $T=\mathcal P_f$ to be the (finite) power-set functor, 
then its posetification is the (finitely generated) convex power-set functor, with the Egli-Milner order (see Example~\ref{ex:posetif_powerset}, but also~\cite{calco2011,kurz+velebil:presentations}). 
\item The collection of (finitary) 
Kripke polynomial endofunctors of $\Set$ 
is inductively defined as follows: 
$
T
::= 
\id \mid 
T_{X_0} \mid 
T_0 + T_1 \mid 
T_0 \times T_1 \mid 
T^A\mid \mathcal P_f
$, 
where $T_{X_0}$ denotes the constant functor to the set $X_0$; 
$T_0+T_1$ is the pointwise coproduct $X\mapsto T_0X+T_1X$;  
$T_0\times T_1$ is the pointwise product $X\mapsto T_0X\times T_1X$; 
and $T^A$ denotes the pointwise exponent functor $X\mapsto (TX)^A$, 
with the set $A$ being finite.

\noindent We have just mentioned above that the posetification of the identity functor is again the identity, while for the constant functor $T_{X_0}$ it is an easy exercise to check that the posetification is again a constant functor, this time to the discrete poset $DX_0$; the posetification of the coproduct $T_0+T_1$ maps a poset $X$ to the coproduct (in the category of posets) $T'_0X+T'_1X$, where $T'_0$ and $T'_1$ denote the posetifications of $T_0$, respectively $T_1$; and similarly for the product functors. Finally, the posetification of the exponent functor $T^A$ is ${T'}^{DA}$, where $T'$ stands for the posetification of $T$. 
 
\item Consider now the finitary probability functor $\mathsf{Prob}:\Set\to \Set$, given on objects by 
\[
\mathsf{Prob}(X)
=
\{p:X\to[0,1]\mid \sum_{x\in X} p(x)=1, \ \mathsf{supp}(p) <\infty\},
\]
where $\mathsf{supp}(p)=\{x\in X\mid p(x)\neq 0\}$, 
and by
\[
\mathsf{Prob}(f)(p)(y)
=
\sum_{y=f(x)} p(x) \ , \ \mbox{ for a function } f:X\to Y.
\] 
on morphisms.
Recall that $\mathsf{Prob}$ preserves weak pullbacks~\cite{vinkrutten}, thus its posetification $\mathsf{Prob}'$ can be described using the relation lifting as in Remark~\ref{wpb-exsq}. In fact, for the probability functor, it happens that the relation lifting of a partial order is not just a preorder, but even a partial order~\cite{baier}. Henceforth for a poset $X$, $\mathsf{Prob}'(X)$ has the underlying set $\mathsf{Prob}(X_0)$, ordered as follows: 
for $p, p'\in \mathsf{Prob}(X_0)$, $p\leq p'$ if and only if there 
is some $\omega\in \mathsf{Prob} (X_0\times X_0)$ such that 
$\sum_{x'\in X}\omega(x,x')=p(x)$ and 
$\sum_{x\in X}\omega(x,x')=p'(x')$, and 
$\omega(x,x')>0\Rightarrow x\leq x'$.
\end{enumerate}
\end{example}

\smallskip
%========================================================%
 
\subsection{Characterising functors $\Pos\to\Pos$ in discrete arities}
\label{sec:posetification2} \

\medskip\noindent
Recall from Proposition~\ref{poset=refl_coins_discrete} that we have denoted, for each poset $X$, by~\eqref{eq:poset_coins} the diagram 
\[
\xymatrix{
DX_1
\ar@<.5ex> [0,1]^-{Dd^1}
\ar@<-.5ex> [0,1]_-{Dd^0}
&
DX_0
}
\]
of discrete posets, where $X_0$ is the set of elements
of $X$, $X_1$ is the set of all pairs $(x,x')$ with
$x\leq x'$ in $X$, while the maps $d^0$, $d^1$ are the obvious
projections.

\begin{theorem}\label{thm:pres-nerves}
For $T':\Pos\to\Pos$, the following are equivalent:
\begin{enumerate}
\item
There exists a functor $T:\Set\to\Set$ such that
$T'\cong \Lan_{D}{(DT)}$, i.e., $T'$ is a posetification
of $T$.
\item
$T'$ preserves discrete posets and coinserters of
all diagrams~\eqref{eq:poset_coins}.
\end{enumerate}
\end{theorem}
\begin{proof}
We prove first that the coinserters of diagrams~\eqref{eq:poset_coins} form the density presentation of $D:\Set\to\Pos$ in the sense of \cite[Section~5.4]{kelly:enriched}. Indeed, all coinserters of~\eqref{eq:poset_coins} exist in $\Pos$, the category
$\Pos$ is the closure of $\Set$ under these coinserters,
and the coinserters of~\eqref{eq:poset_coins} are preserved
by the functor $\Pos(DS,-):\Pos\to \Pos$. To see the latter, observe that for any set $S$, the poset
$\Pos(DS,X)=X^S$ is a coinserter of
\begin{equation}\label{eq:X^S}
\xymatrix{
(X_1)^S
\ar@<.5ex> [0,1]^-{(d^1)^S}
\ar@<-.5ex> [0,1]_-{(d^0)^S}
&
(X_0)^S
}
\end{equation}
We prove now that~(1) implies~(2). Since $TD\cong DT'$
holds, $T$ preserves discrete posets. By~Equation~\eqref{eq:X^S}, the
collection of all coinserters of~\eqref{eq:poset_coins} forms a density
presentation of $D$, hence by \cite[Theorem~5.29]{kelly:enriched},
$T$ preserves coinserters of all diagrams in~\eqref{eq:poset_coins}.

\medskip\noindent
(2) implies (1). Since $T'$ is assumed to preserve
discrete posets, we may assume that $T'D\cong DT$
for some functor $T:\Set\to\Set$. Furthermore, by
\cite[Theorem~5.29]{kelly:enriched}, $T'\cong\Lan_{D}{(T'D)}$
holds. That is, $T'\cong\Lan_{D}{(DT)}$ holds.
\end{proof}

\medskip\noindent Recall that by Definition~\ref{def:discretearities}, a functor has a presentation in discrete arities if it is of the form $\Lan_{D}{H}$ for some $H:\Set\to\Pos$. Then we  can drop in the theorem above the requirement that $T$ preserves discrete posets to obtain the following.

\begin{theorem}\label{thm:discretearities}
A functor $\Pos\to\Pos$ has a presentation in discrete arities 
if and only if it preserves coinserters
of all diagrams~\eqref{eq:poset_coins}.
\end{theorem}

\noindent Similarly, recalling Definition~\ref{def:strongly_finitary} of a strongly finitary functor, one has: 

\begin{theorem}
\mbox{}\hfill
\begin{enumerate}
\item A functor $\Pos\to \Pos$ is strongly finitary and preserves 
discrete posets if and only if it is the posetification of 
a finitary functor $\Set\to \Set$. 
\item A functor $\Pos\to \Pos$ is strongly finitary if and only 
if it is a left Kan extension of a finitary functor 
$\Set\to\Pos$.
\end{enumerate}
\end{theorem}

\noindent Having in the above characterized functors $\Pos\to \Pos$ in discrete arities, we now turn to a special property that these have. First, recall that every endofunctor on $\Set$ preserves surjections~\cite{trnkova69} (assuming the axiom of choice). We shall establish below the correspondent result for $\Pos$, recalling again Definition~\ref{def:discretearities}. 

\begin{proposition}\label{prop:T'-pres-monot-surj}
Let $T':\Pos\to\Pos$ have a presentation in discrete arities.
Then $T'$ preserves monotone surjections between posets. 
\end{proposition}

\begin{proof}
Let $c:X \to Y$ be a surjective monotone map between posets. Then it is easy to see that $c$ is the (reflexive) coinserter of the comma object of $c$ with itself
\begin{equation}\label{surj=coins_comma}
\xymatrix{
P
\ar@<-0.3ex>[r]_{d^1} \ar@<0.3ex>[r]^{d^0}
& 
X
\ar[r]^c
&
Y}
\end{equation}
where $P$ is the poset of all pairs $(x,x')$ such that $c(x)\leq c(x')$, ordered component-wise, and $d^0,d^1$ are the canonical projections. 

We want to show that $T'c$ is again surjective. In fact, we shall see more: that $T'$ preserves the coinserter~\eqref{surj=coins_comma}. As $T'$ is a left Kan extension along $D:\Set \to \Pos$, according to \cite[Theorem~5.29]{kelly:enriched}, it is enough to check that~\eqref{surj=coins_comma} is $D$-absolute, that is, that \eqref{surj=coins_comma} is preserved by $\Pos(DS,-)$ for every set $S$. To see the latter, let $S$ be an arbitrary set and form the diagram
\[
\xymatrix{  
P^S 
\ar@<-0.3ex>[r]_{(d^1)^S} \ar@<0.3ex>[r]^{(d^0)^S}
& 
X^S
\ar[r]^{c^S}
&
Y^S}
\]
To prove that it is a coinserter, consider any monotone map $h : X^S\to Z$, with $h\circ (d^0)^S \leq h \circ (d^1)^S $. Using the surjectivity of $c$, define $k : Y^S \to Z$ by $k((y_i)_{i \in S}) = h((x_i)_{i\in S})$, where $c((x_i)_{i\in S}) = (y_i)_{i\in S}$. Using the construction of the comma object $P$, one can easily check that the above does not depend on the choice of $(x_i)_{i\in S}$,
%Indeed, suppose that e(x'_i) = e(x_i). Then both (x'_i,x_i) and (x_i,x'_i) are in P^S. Thus h (x_'i) \leq h (x_i) and h (x_i) \leq h (x'_i). Thus h (x_'i) = h (x_i).
and that the map $k$ thus defined is indeed monotone. Now the universal property of coinserters follows easily. 
%Indeed, suppose that e(x'_i) = e(x_i). Then both (x'_i,x_i) and (x_i,x'_i) are in P^S. Thus h (x_'i) \leq h (x_i) and h (x_i) \leq h (x'_i). Thus h (x_'i) = h (x_i).
\end{proof}
%\end{ab}

\begin{example}
We give an example of a $\Pos$-functor which does not preserve surjections (and the definition of which involves a non-discrete arity). Let $T':\Pos\to\Pos$ be the functor mapping a poset $X$ to the poset of monotone maps $[\Two,X]$ (equivalently, it could be written more intuitively as $X^\Two$, the poset of \emph{ordered pairs} in $X$, with the component-wise order). The fact that $[\Two,X]$ only contains monotone maps has as a consequence that the surjection $f:D2\to \Two $ mapping each of the $0,1$ to itself, see below
\[
\xymatrix@R=3pt@C=20pt{
&&&& 1 \\
0  \quad 1  & \ar@{->>}[rr]^f  &&& \\
&&&& 0 \ar@{-}[uu]
}
\]
is not preserved by $T'$. Indeed $T'D2$ has two elements, while $T'\Two$ has three elements.
\end{example}

%========================================================%

\section{Presenting functors on ordered varieties}\label{ak} 

\medskip\noindent
Coming back to the introduction, we remind the reader that our overall strategy is---starting with a functor $T:\Set\to\Set$ for the type of coalgebras---to obtain from $T$ the Boolean logic $L:\BA\to\BA$ by duality, and to obtain from the posetification $T'$ of $T$, again by duality, the positive logic $L':\DL\to\DL$. The relationship between $L$ and $L'$ will be studied in the next section. Here, we are going to make sure that the functors $L$ and $L'$ obtained by abstract categorical constructions actually do have concrete presentations by operations and equations and thus correspond indeed to modal extensions of Boolean and positive propositional logic. 

\medskip\noindent
In the case of $L$, assuming that $L$ preserves sifted colimits, this is known already from~\cite{kurz-rosicky:strong} and we shall recall it below.
In the case of $L'$, we need to prove the enriched analogue 
of~\cite{kurz-rosicky:strong}, which we shall obtain following the enriched generalization of~\cite{kurz-rosicky:strong} 
given in~\cite{kurz+velebil:presentations}. In particular, this enriched generalization will guarantee that  $L'$ can be presented by \emph{monotone} operations.
As a final twist, this enriched generalization would give us 
a presentation of $L'$ using \emph{inequations} over general ordered 
varieties. Therefore, it is important for us to show that, owing to the special nature of $\DL$, the enriched functor $L'$ can equally be presented as the underlying {\em ordinary\/} 
functor $L'_o$, which in turn has a presentation that does not rely on inequations.

%========================================================%

\subsection{Equational presentations of functors} \label{sec:presentations}

We have seen, in Example~\ref{exle:Lpcal}, a presentation of  a functor $L:\BA\to\BA$ and in Example~\ref{exle:L'pcal} a presentation of  a functor $L':\DL\to\DL$. Whereas it may be clear from these examples what we mean by a presentation, it is worth spending the effort to give a formal definition.

In what follows, $\mathscr A$ will denote a variety of algebras for a finitary signature. By a slight abuse of notation, we shall use 
the same notation as in case of the variety $\BA$ for the (monadic) adjunction 
\[
F\dashv U:\mathscr A\to \Set
\]
Denote by $\Sigma_n$ the set of $n$-ary modal operators, and by $\Gamma_n$ the set of equations in $n$ free variables. For instance, in Example~\ref{exle:Lpcal} we have $\Sigma_1=\{\Box\}$ and $\Sigma_n=\emptyset$ for $n\not=1$, and $\Gamma_0=\{(\Box \top,\top)\}$, $\Gamma_2=\{(\Box(a\wedge b),\Box a \wedge \Box b)\}$ and  $\Gamma_m=\emptyset$ for $m\not=0,2$. Given a signature $\Sigma = (\Sigma_n)_{n<\omega}$, we write $\hat\Sigma:\Set\to\Set$ for the corresponding {\em polynomial\/} functor 
$$
X\mapsto \coprod_{n<\omega}\Set(n,X)\bullet \Sigma_n
$$ 
where $\bullet$ denotes the copower, see Section~\ref{sec:3.C}. 
Observe that with this notation, in Example~\ref{exle:Lpcal} we have that $\Gamma_n\subseteq UF\hat\Sigma U Fn\times UF\hat\Sigma U Fn$ (interpret $UFX$ as the set of Boolean terms on $X$-generators).

\begin{definition}\cite[Definition~6]{bons-kurz:fossacs06}
\label{def:presentation}
A functor $L$ on a variety $\mathscr A$ has a 
\emph{presentation by operations and equations}, or, shortly, 
a \emph{presentation}, if there are signatures $\Sigma$ and $\Gamma$, 
with  $\Gamma _n\subseteq UF\hat\Sigma U Fn\times UF\hat\Sigma U Fn$,
such that for all $A\in\mathcal A$ the following diagram, where $n$
ranges over natural numbers and $v$ ranges over all valuations $Fn\to A$ (of $n$-variables in $A$)
\begin{equation}\label{eq:presentation1}
\xymatrix{
F\hat\Gamma n \ar@<.5ex>[r] \ar@<-.5ex>[r] & F\hat\Sigma U Fn \ar[r]^{F\hat\Sigma U v} & F\hat\Sigma UA \ar[r] & LA
}
\end{equation}
is a joint coequalizer.
\end{definition}

\noindent 
Recall that a joint coequalizer of a family of parallel pairs 
with common codomain is an arrow which is a coequalizer of each 
pair of maps in that family.

\medskip\noindent The elements of the sets $\Gamma_n$, $n<\omega$, are often called the \emph{axioms}, or \emph{equations} of the presentation.

\begin{remark}[\textbf{Axioms of rank $1$}]
We see that the format of the equations (i.e., the elements of $\Gamma _n$) requires them to be pairs in $UF\hat\Sigma U Fn\times UF\hat\Sigma U Fn$, that is, every variable must be under exactly one modal operator. Such equations are called equations,  or axioms, \emph{of rank $1$}. For example, if we wanted to \emph{extend} $\DL$ by negation (thinking of negation as a unary modal operator), then $\neg(a\wedge b)=\neg a\vee\neg b$ and $\neg 1 = 0$ are equations of rank $1$, but $a\wedge\neg a = 0$ is not. The importance of equations of rank $1$ is that they capture algebras for a \emph{functor} (as opposed to general equations which correspond to algebras for a monad), see Theorem~\ref{th:5.4} below.  
\end{remark}

\noindent For proofs of the following proposition and theorem
see~\cite[Theorem~4.7]{kurz-rosicky:strong}.

\begin{proposition}
A functor $L$ on a variety $\mathscr A$ 
has a presentation if and only if there are polynomial functors 
$\hat\Sigma, \hat\Gamma:\Set\to\Set$ such that 
$L$ is a coequalizer
\begin{equation*}\label{eq:presentation}
\xymatrix{
F\hat\Gamma U \ar@<.5ex>[r] \ar@<-.5ex>[r] & F\hat\Sigma U  \ar[r] & L
}
\end{equation*}
in the category of endofunctors of $\mathscr A$.
\end{proposition}

\medskip\noindent Recall that any ordinary variety  $\mathscr{A}$ 
can be presented by a signature $\Sigma_{\mathscr{A}}$ 
and equations $E_{\mathscr{A}}$. For instance the variety $\DL$ 
is presented by the constants $\bot$, $\top$, the binary operations 
$\wedge$, $\vee$ and the usual equations defining distributive 
lattices, see e.g.~\cite{davey-priestley}.

\begin{theorem}
\label{th:5.4}
Let $L$ be an endofunctor of a variety $\mathscr{A}$. Let 
$\mathscr A$ be presented by a signature $\Sigma_{\mathscr{A}}$ 
and equations $E_{\mathscr{A}}$. Then:
\begin{enumerate}
\item If $L$ has a presentation $\langle \Sigma,\Gamma\rangle$, 
then the category of $L$-algebras is isomorphic to the category 
of algebras for the signature $\Sigma_{\mathcal{A}}+\Sigma$ 
satisfying the equations  $E_{\mathcal{A}}$ and $\Gamma$.
\item 
An endofunctor $L$ of a variety $\mathscr A$ has a presentation  
if and only if it preserves ordinary sifted colimits. 
\end{enumerate}
\end{theorem}

\noindent This theorem gives a bijection between endofunctors 
$L$ of a variety $\mathscr A$ that preserve sifted colimits 
and logics \emph{extending} $\mathscr A$  
by `modal operators' and axioms of rank $1$. The theorem enables 
us to investigate such logics using purely category theoretic means.

%========================================================%

\subsection{Equational presentations of locally monotone functors} 
\label{sec:presentations2}

For the purposes of our investigations, we are interested in modal logics extending $\DL$, given by rank $1$ axioms of \emph{monotone} operations. 
While $U:\DL\to\Set$ is certainly finitary and monadic (since $\DL$ is an ordinary variety of algebras), it is also the case that the natural forgetful functor $U':\DL\to\Pos$, mapping a distributive lattice to its carrier equipped with the lattice order, exhibits $\DL$ as an {\em ordered\/} variety. 

\medskip\noindent For now, let us be slightly more general and consider an ordered variety
\[
F'\dashv U':\mathscr A\to \Pos
\]
By an \emph{ordered signature} $\Sigma'$ we shall mean 
a family of \emph{posets} $\Sigma'=(\Sigma'_n)_{n<\omega}$. 
Denote by $\tilde\Sigma':\Pos\to\Pos$ the corresponding {\em polynomial\/} 
functor 
$$
X\mapsto \coprod_{n<\omega}\Pos(Dn,X)\bullet \Sigma'_n. 
$$
Recall that $\bullet$ denotes the copower, which in the case above is just the cartesian tensor product in $\Pos$. 

\medskip \noindent In the following, we shall call a functor $\Pos\to\Pos$ 
\emph{polynomial} only if it is of the form $\tilde\Sigma'$, for 
some ordered signature $\Sigma'$. 
Notice that a polynomial functor only employs 
\emph{discrete arities}, and that if $\Sigma'=D\Sigma$ 
for some (necessarily unique!) $\Set$-signature $\Sigma$, 
then $\tilde\Sigma'$ is the posetification of $\hat\Sigma$ 
in the sense of Definition~\ref{def:extension}. 

\begin{definition}\label{def:presentation2}
A functor $L'$ on an ordered variety $\mathscr A$ 
has an \emph{ordered presentation in discrete arities}, or, 
shortly, an \emph{ordered presentation}, if there are ordered 
signatures $\Sigma'$ and $\Gamma'$, with $\Gamma' _n\subseteq U'F'\tilde \Sigma' U' F' D n\times U'F'\tilde\Sigma' U' F' D n$ for each natural number $n$, such that for all $A$ in $\mathscr A$, the following diagram, where $n$ 
ranges again over natural numbers and $v$ ranges over all 
valuations $F'Dn\to A$ (of $n$-variables in $A$)
\begin{equation}\label{eq:presentation2}
\xymatrix@C=40pt{
F'\tilde\Gamma' Dn \ar@<.5ex>[r] \ar@<-.5ex>[r] & F'\tilde\Sigma' U' F'Dn \ar[r]^{\ \ F'\tilde\Sigma' U' v} & F'\tilde\Sigma' U'A \ar[r]^{q}
 & L'A
}
\end{equation}
is a joint coequalizer.
\end{definition}

\noindent In the definition above it is not important to 
allow $\Gamma'_n$ to be posets. 
On other hand, for a general variety 
$\mathscr A$ it is important to allow the $\Sigma'_n$ to be posets 
(which is the reason why we can use a coequalizer in \eqref{eq:presentation2} instead of a coinserter). 
Then again, for $\mathscr A=\DL$ we can take the $\Sigma'_n$ 
discrete since for $\DL$ the order is equationally definable, 
as will be discussed in detail in Section~\ref{pres_DL}.

\begin{remark}\label{rem:presentations2}
An ordered presentation is monotone. In detail, let 
$\alpha:L'A\to A$ be an algebra. Consider an operation
$\sigma\in\Sigma'_n$ and $a,a':Dn\to U'A$ with $a\le a'$, 
that is, $a_i\le a'_i$ for all $1\le i \le n$. We have to show 
that $\sigma(a)\le\sigma(a')$ holds in the $L'$-algebra $(A,\alpha)$. But this is equivalent to the obvious inequality $U'\alpha\circ q^\flat(\sigma,a)\le U'\alpha\circ q^\flat(\sigma,a')$ where $q^\flat:\tilde\Sigma' U'A\to U'L'A$ is the adjoint transpose of the quotient map $q:F'\tilde \Sigma'U'A \to L'A$ and $(\sigma,a)$ and $(\sigma,a')$ are pairs in $\Sigma'_n\bullet\Pos(Dn,U'A)=\Sigma'_n\times\Pos(Dn,U'A)$.
\end{remark}

\begin{proposition}\label{prop:presentation2}
A locally monotone functor $L'$ on an ordered variety 
$\mathscr A$ has an ordered presentation if and only if there are polynomial functors $\tilde\Sigma, \tilde\Gamma:\Pos\to\Pos$ such that $L$ is a coequalizer
\begin{equation*}
\xymatrix{
F'\tilde\Gamma U' \ar@<.5ex>[r] \ar@<-.5ex>[r] & F'\tilde\Sigma U'  \ar[r] & L'
}
\end{equation*}
in the category of locally monotone endofunctors.
\end{proposition}

\begin{proof} See \cite[Theorem~3.18]{kurz+velebil:presentations} where the more general situation was considered.
\end{proof}

\begin{theorem}\label{thm:eq-pres}
Let  $L':\mathscr A\to\mathscr A$ be a locally monotone functor 
on an ordered variety. Then $L'$ preserves $\Pos$-sifted colimits 
if and only if it has an ordered presentation in discrete arities. 
\end{theorem}
\begin{proof} 
The proof is essentially contained in~\cite{kurz+velebil:presentations}, but for the reader's convenience, we shall spell out some of the details.

\smallskip\noindent We denote by $\mathscr A_\ff$ the full subcategory of $\mathscr A$ spanned by the algebras which are free on finite discrete posets and by $F'_f:\Set_f\to \mathscr A_\ff$ the corresponding domain-codomain restriction of the composite $F' D:\Set\to \Pos\to\mathscr A$. Then we follow the next steps: 
\begin{enumerate}
\item 
Observe first that 
$[F'_f,U']:[\mathscr A_\ff,\mathscr A]\to [\Set_f,\Pos]$,
sending a functor $L':\mathscr A_\ff\to\mathscr A$ to the 
composite $U' L' F'_f$, is of descent type. That is, $[F'_f,U']$ is right adjoint and each component of the counit of the corresponding adjunction is a coequalizer~\cite[Lemma~3.14]{kurz+velebil:presentations}. 
\item
The functor $[E,-]:[\Set_f,\Pos]\to [|\Set_f|,\Pos]$, where 
$|\Set_f|$ is the discrete category of finite sets 
and $E:|\Set_f|\to\Set_f$ is the inclusion, is monadic 
(again, this follows from
Lemma~3.14 of~\cite{kurz+velebil:presentations}).
\item
Consequently, the composite
$$
\xymatrixcolsep{3pc}
\xymatrix{
[\mathscr A_\ff,\mathscr A]
\ar[0,1]^-{[F'_f,U']}
&
[\Set_f,\Pos]
\ar[0,1]^-{[E,-]}
&
[| \Set_f |,\Pos]
}
$$
is of descent type~\cite[Theorem~3.18]{kurz+velebil:presentations}.
\end{enumerate}
\noindent Therefore, every locally monotone functor 
$L':\mathscr A_\ff\to\mathscr A$ (i.e., every $L'$ preserving 
sifted colimits) 
arises as a coequalizer as 
in Proposition~\ref{prop:presentation2}, that is, it admits an ordered presentation in discrete arities. 
\end{proof}

%========================================================%

\subsection{Ordinary and ordered presentations of functors on $\BA$}\label{discreteBA}

Let $\mathscr A$ be a $\Pos$-enriched category with 
discretely ordered hom-posets, such as $\BA$. Then, as we 
are going to show now, there is no essential difference in 
the presentations according to Sections~\ref{sec:presentations2} 
and~\ref{sec:presentations}.

\medskip\noindent Before coming to functors on varieties, let us clarify when ordinary varieties are ordered varieties with discrete hom-sets and vice versa. Recall that we wrote $C\dashv D:\Set\to\Pos$ for the adjunction in which $D$ is the discrete functor and $C$ the connected components functor.

\begin{proposition}
Let $F'\dashv U':\mathscr A\to\Pos$ be an ordered variety. It has discretely ordered hom-posets if and only if any of the following equivalent conditions are satisfied.
\begin{enumerate}
\item $U'$ factors through $D:\Set\to\Pos$.
\item $U'\cong DCU'$.
\item $\eta_{U'}: U'\to DCU'$ is an isomorphism, where $\eta$ is the unit of the adjunction $C\dashv D$.
\end{enumerate}
If any of the above conditions is satisfied then also $F'$ factors through $C:\Pos\to\Set$ via $F'\cong F'DC$. Moreover, $F'D\dashv CU':\mathscr A\to\Set$ is monadic.
\end{proposition}

\noindent Conversely, if $F\dashv U:\mathscr A\to\Set$ is an ordinary variety and the only order on algebras in $\mathscr A$ making all operations monotone is the trivial discrete order (as it is the case in $\BA$), then 
$FC\dashv DU:\mathscr A\to\Pos$ is an ordered variety, see~\cite{kurz+velebil:varieties}.

\begin{proposition}
Let $\mathscr A\to\Pos$ be a variety with discretely ordered 
homsets and let $L:\mathscr A\to \mathscr A$ be a 
(necessarily locally monotone) functor. Then the functor 
$L$ preserves ordinary sifted colimits if and only if $L$ preserves $\Pos$-enriched sifted colimits. Moreover, $\langle D\Sigma,D\Gamma\rangle $ 
is an ordered presentation of $L$ if and only if $\langle \Sigma,\Gamma\rangle$ 
is a presentation of $L$ and $\langle \Sigma',\Gamma'\rangle $ is 
an ordered presentation of $L$ if and only if $\langle C\Sigma',C\Gamma'\rangle$ 
is a presentation of $L$.
\end{proposition}

\noindent The proposition above guarantees that for a functor $L:\BA\to\BA$, it does not matter whether we consider it as an ordinary functor on the variety $\BA$, or whether we consider it as a locally monotone functor on the ordered variety $\BA$.

%========================================================%

\subsection{Ordinary and ordered presentations of functors on $\DL$}\label{pres_DL}

The aim of this section is to show that not only does an endofunctor 
of $\DL$ have a presentation by operations and equations 
if and only if it preserves ordinary sifted colimits, but also that 
a functor has a presentation by \emph{monotone} operations 
and equations if and only if it preserves \emph{enriched} sifted colimits.

\medskip\noindent We begin with the following:

\begin{proposition}
If $\mathscr A$ is an ordered variety and $L':\mathscr A\to\mathscr A$ 
is a locally monotone functor which
preserves enriched sifted colimits, then 
the underlying ordinary functor
$L'_o:\mathscr A_o\to\mathscr A_o$
preserves ordinary sifted colimits.
\end{proposition}

\begin{proof}
By~\cite[Theorem~6.9]{kurz+velebil:varieties}
we know that $\mathscr A$ is a free cocompletion by enriched 
sifted colimits of the full subcategory 
$\mathbf J':\mathscr A_{\ff}\hookrightarrow\mathscr A$ spanned by free algebras 
on finite and discrete sets of generators. 

\noindent Furthermore, $\mathbf J':\mathscr A_{\ff}\hookrightarrow\mathscr A$
has the density presentation consisting of the three classes below:
\begin{enumerate}
\item reflexive coinserters,
\item (conical) filtered colimits, i.e., by colimits weighted 
by $\varphi:\mathscr D^\op \to \Pos$ where $\mathscr D$
is ordinary filtered category and $\varphi$ 
is the constant functor at the one-element poset,
\item reflexive coequalizers.
\end{enumerate}
The reason is that we can
\begin{enumerate}
\item use coinserters of truncated nerves to create algebras,
 free on any finite poset,
\item use (conical) filtered colimits to obtain free algebras on any
 poset,
 \item use a reflexive coequalizer (that is, a canonical presentation) to obtain
 any algebra.
 \end{enumerate}
Hence, we know that $L':\mathscr A\to\mathscr A$ 
preserves enriched sifted colimits if and only if $L'$ preserves colimits in~(1),
(2) and~(3).
Since colimits in~(2) and~(3) are conical, they are preserved
by $L'_o$. But $\mathscr A$ is $\Pos$-cocomplete, being an ordered variety, hence $\mathscr A_o$ is $\Set$-cocomplete, \linebreak according to~\cite[\mbox{Section~3.8}]{kelly:enriched}.
And a functor between ordinary cocomplete categories preserves 
sifted colimits if and only if it preserves ordinary filtered colimits 
and reflexive coequalizers~\cite[Theorem~7.7]{ARV}.
\end{proof}

\medskip\noindent The above proposition, together with Theorem~\ref{th:5.4}, imply that in case 
$\mathscr A_o$ is an ordinary variety, the underlying ordinary functor $L'_{o}$ also has a presentation by operations and equations. 

\begin{example}
Let $L':\DL\to\DL$ be the locally monotone functor presented by one unary operation, written as $\Box$, and no equations. It follows from the proposition that monotonicity of $\Box$ is equationally definable. Explicitly, the induced equational presentation of $L'_o$ can be given by 
$$\Box a \wedge \Box (a\vee b) = \Box a$$
Of course, the proposition only tells us that all finitary equations valid for a monotone $\Box$ together present $L'_o$. But it is not difficult to check that the equation above is enough to force $\Box$ to be monotone.
\qed
\end{example}

\medskip\noindent Conversely, if $\mathscr A_o$ is an ordinary variety, it makes sense 
to ask how a presentation of a functor on $\mathscr A_o$ 
induces a presentation of a functor on $\mathscr A$. 
We are thinking of a situation such as the one of Example~\ref{exle:L'pcal}, where the functor $L'$ is defined equationally on the ordinary variety $\DL_o$, but can also be seen as a locally monotone functor on the ordered variety $\DL$.
Consider thus an~ordered variety $F'\dashv U':\mathscr A\to\Pos$. 
By slight abuse of notation, we let $\mathscr A_o$ 
to stand both for the underlying ordinary category of $\mathscr A$, 
and for the $\Pos$-subcategory which has the same objects and arrows 
as $\mathscr A$ but discrete homsets, and we put 
$\tilde D:\mathscr A_o\to \mathscr A$ to be the (necessarily
locally monotone) inclusion. 
Assume that $\mathscr A_o$ is an ordinary variety. 
Then to say that the following diagram 
$$
\xymatrix{
\mathscr A_o\ar[r]^{\tilde D} & \mathscr A \\
\Set \ar[r]_{D} \ar[u]^{F}&  \Pos \ar[u]_{F'}
}
$$
commutes defines $F$. 
Let $\langle\Sigma, \Gamma\rangle$ be a presentation
of a functor $L_o:\mathscr A_o\to\mathscr A_o$ by the coequalizer \eqref{eq:presentation1} of Definition~\ref{def:presentation}. 
It induces a presentation $\langle \Sigma' = D\Sigma, \Gamma' = D\Gamma\rangle$ 
of a functor \linebreak $L':\mathscr A\to \mathscr A$ by the coequalizer \eqref{eq:presentation2} of Definition~\ref{def:presentation2}. 
Since $\tilde D$ preserves the coequalizer \eqref{eq:presentation1}, we obtain for all $A\in\acal$
the dotted arrow in the diagram

\begin{equation}\label{eq:AoA}
\vcenter{
\xymatrix@C=40pt{
\tilde DF\hat\Gamma n \ar@<.5ex>[r] \ar@<-.5ex>[r] \ar[d]
& \tilde DF\hat\Sigma U Fn \ar[r]^{\tilde D F\hat\Sigma U v} \ar@{-->}[d]
& \tilde DF\hat\Sigma UA \ar[r] \ar@{-->}[d]
& \tilde DL_oA\ar@{..>}[d]
\\
F'\tilde\Gamma' Dn \ar@<.5ex>[r] \ar@<-.5ex>[r] 
& F'\tilde\Sigma' U' F'Dn \ar[r]^{\ \ \tilde D F'\tilde\Sigma' U' v} 
& F'\tilde\Sigma' U'A \ar[r]^{}
& L'A
}}
\end{equation}

\medskip\noindent
Due to $\tilde DF=F'D$ and $D\hat\Gamma=\tilde\Gamma' D$, the left-most vertical arrow is an isomorphism. Even though  the two vertical dashed arrows in the middle need not be order-reflecting, they are surjective, which implies that also the the dotted vertical arrow on the right is onto. But since the lower row may have more inequalities, the dotted arrow need not be bijective, see Example~\ref{exle:AoA} below.

\begin{definition}
We say that $\langle\Sigma, \Gamma\rangle$ is  a \emph{presentation by monotone operations and equations}, or, shortly, a \emph{monotone presentation},  if the dotted arrow~\eqref{eq:AoA} is an isomorphism for all $A\in\acal$. 
\end{definition}
\noindent This terminology is justified by Remark~\ref{rem:presentations2}, according to which $\langle D\Sigma, D\Gamma\rangle$ is a presentation by monotone operations.

\begin{example}\label{exle:AoA}
The presentation of Example~\ref{exle:L'pcal} is a presentation by monotone operations, since to say that $\Box$ preserves meets and that $\Diamond$ preserves joins forces $\Box$ and $\Diamond$ to be monotone. 
On the other hand, if these axioms had been omitted from the presentation, the resulting presentation would not have been monotone.
\end{example}

\medskip\noindent To summarize, given an ordered presentation  $\langle\Sigma', \Gamma'\rangle$ of a functor $L':\mathscr A\to\mathscr A$ on an ordered variety in the sense of Definition~\ref{def:presentation2}, there is a monotone presentation by operations and equation $\langle \Sigma,\Gamma\rangle$ of the underlying ordinary functor $L'_o$ if $\mathscr A_o\to\Pos_o\to\Set$ is a variety. This is due to the following result.

\begin{theorem}
Let $L$ be an endofunctor on a category $\mathscr A$ that is both 
an ordered and an ordinary variety. If $L$ has an ordered presentation,
then it has a presentation by monotone operations and equations. 
\end{theorem}
\begin{proof}
To say that $\mathscr A$ is both an ordered and an ordinary variety 
is to say that $\mathscr A$ comes equipped with a forgetful functor $\mathscr A\to\Pos$ so that $\mathscr A\to\Pos$ is an ordered variety and $\mathscr A_o\to\Pos_o\to\Set$ is an ordinary variety. 
If $L$ has an ordered presentation then it preserves enriched sifted colimits, hence $L_o$ preserves ordinary sifted colimits, hence $L_o$ has a presentation. 
\end{proof}

\medskip\noindent We can now conclude what we shall need to know about functors on $\DL$.

\begin{theorem}\label{thm:presentmonop}
For a locally monotone functor $L':\DL\to\DL$ the following are equivalent:
\begin{enumerate}
\item $L'$ has a presentation by monotone operations and equations.
\item $L'$ preserves $\Pos$-enriched sifted colimits. 
\item $L'$ is the $\Pos$-enriched left Kan extension of its restriction to discretely finitely generated free distributive lattices.
\end{enumerate}
\end{theorem}

\noindent
As in Proposition~\ref{prop:mathbf L}, we now obtain that

\begin{corollary}\label{cor:L'isboldL'}
  If $T'$ is the the convex powerset functor, then the functor $L'$ of
  Example~\ref{exle:L'pcal} is isomorphic to the sifted colimits preserving functor $\mathbf L'$ whose restriction to $\DL_\ff$ is $P'T'^\op S'$ as in~Equation~\eqref{eq:Ldelta'}.
\end{corollary}

%========================================================%

\section{Positive coalgebraic logic}\label{sec:poscoalglog}

\medskip\noindent
The reader might find it useful to consult Section~\ref{sec:overview} first, even though it relies on some notation introduced in the next two subsections.

%========================================================%

\subsection{Morphisms of logical connections}\label{sec:morph}

We recall the logical connections (dual adjunctions, see~\cite{kurzvelebil:acs}) mentioned in Section~\ref{sec:coalgebras} between sets and Boolean algebras, and between posets and distributive lattices. Both are considered to be $\Pos$-enriched, where for the first logical connection the enrichment is discrete. They are related as follows: 
\begin{equation}
\label{morphconnect}
\vcenter{
\xymatrix@C=60pt@R=15pt{
\Set^{op} 
\ar @{} [r] |-{\perp} 
\ar@/_/[r]_-{P}
\ar[dd]_{D^{\opp}} 
& 
\BA
\ar@/_/[l]_-{S}
\ar[dd]^{W} 
\\
\\
\Pos^{op} 
\ar @{} [r] |-{\perp} 
\ar@/_/[r]_-{P'}
& 
{\DL}\ar@/_/[l]_-{S'}
}
}
\end{equation}
In the top row of the above diagram, we recall again for the reader's convenience that $P$ is the contravariant powerset functor, while $S$ maps a Boolean algebra to its set of ultrafilters. 
The bottom row has $P'$ mapping a poset to the distributive lattice of its 
upsets, and $S'$ associating to each distributive lattice the poset of 
its prime filters. 
As for the pair of functors connecting the two logical connections: $D$ was introduced earlier as the discrete functor, while $W$ is the functor associating to each Boolean algebra its underlying distributive lattice. 

\medskip\noindent It is easy to see that the pair $(D^\opp,W)$ is a 
{\em morphism of adjunctions\/} in the sense of~\cite[\textsection~IV.7]{maclane:cwm}.
This means that the following diagrams commute, and that the coherence condition below holds:
\begin{equation}
\label{eq:morph_of_adj}
\vcenter{
\xymatrix{
\Set^\op 
\ar[0,1]^-{P}
\ar[1,0]_{D^\op}
&
\BA
\ar[1,0]^{W}
\\
\Pos^\op
\ar[0,1]_-{P'}
&
\DL
}}
\qquad
\vcenter{\xymatrix{
\BA
\ar[0,1]^-{S}
\ar[1,0]_{W}
&
\Set^\op
\ar[1,0]^{D^\op}
\\
\DL
\ar[0,1]_-{S'}
&
\Pos^\op
}}
\qquad \mbox{$\epsilon' D^\opp = D^\opp\epsilon$}
\end{equation}
where $\epsilon$ and $\epsilon'$ are
the counits of $S\dashv P$ and $S'\dashv P'$, respectively.

\begin{remark}\label{rmk:functor-K}
It will turn out useful later to use not only that $D$ has, as mentioned in Section~\ref{sec:pos}, as a $\Pos$-enriched left adjoint the connected components functor $C$, but also that $W$ has a $\Pos$-enriched right adjoint $K$, mapping a distributive lattice $A$ to the Boolean algebra of complemented elements in $A$, also known as the centre of $A$, see~\cite{birkhoff:lt}. Then the mate of the first square in~\eqref{eq:morph_of_adj} under the above adjunctions, namely $PC^\op \to KP'$, is in fact an isomorphism (to see this, use that the connected components of a poset are precisely its minimal subsets which are both upward and downward closed). 
\end{remark}

%========================================================%

\subsection{Positive coalgebraic logic} 

We shall now expand the propositional logics $\BA$ and $\DL$ by modal operators. We start with an endofunctor $T$ of $\Set$ in the top left-hand corner of~\eqref{morphconnect}, and an endofunctor $T'$ 
of $\Pos$ in the bottom left-hand corner. We are mostly interested in the case where $T':\Pos\to\Pos$ is the posetification of $T$ (see Definition~\ref{def:extension}) and $L:\BA\to\BA$ and $L':\DL\to\DL$ are (the functors of) the associated logics as in~\eqref{eq:LFn} and~\eqref{eq:L'F'n}, in which case we denote the logics by boldface letters $\bf L$ and $\bf L'$. 

\medskip\noindent But some of the following hold under the weaker assumptions that 
$T'$ is an arbitrary extension of $T$, and that $L$ and $L'$ are arbitrary coalgebraic logics for $T$ and $T'$, respectively. 

\medskip\noindent Let therefore $T$ be an endofunctor of $\Set$, and $T'$ be an extension of $T$ to $\Pos$ as in~\eqref{eq:extension}. Logics for $T$, respectively $T'$ are given by functors $L:\BA\to\BA$ and $L':\DL\to\DL$ and by natural transformations 
\[
\delta:LP\to PT^\op \quad \quad \delta':L'P'\to P'{T'}^\op
\]
Intuitively, $\delta$ and $ \delta'$ assign to the syntax given by (presentations of) $L$ and $L'$ the corresponding one-step semantics in subsets, respectively upsets. 
To compare $L$ and $L'$ we need the isomorphism $\alpha:DT\to T'D$ from~\eqref{eq:extension} saying that $T'$ extends $T$, and also the relation $WP=P'D^\op$ from~\eqref{eq:morph_of_adj} (which formalizes the trivial observation that taking all upsets of a discrete set is the same as taking all subsets). 
Referring back to the introduction, we now make the following
definition.

\begin{definition}\label{def:pos-frag}
We say that a logic $(L',\delta')$ for $T'$ is \emph{a positive fragment} of the logic $(L,\delta)$ for $T$, if there is a natural transformation $\beta: L'W\to WL$ with $W\delta \circ \beta P =  P'\alpha^\opp \circ \delta'D^\opp$, or, in diagrams
\begin{equation}
\label{eq:fragment}
\vcenter{
\xymatrix@C=30pt{
\Set^\opp
\ar[0,1]^-{P}
\ar[1,0]_{T^\opp}
&
\BA
\ar[0,1]^-{W}
\ar[1,0]^{L}
\ar@{} [1,-1]|{\swarrow\delta}
&
\DL
\ar[1,0]^{L'}
\ar@{} [1,-1]|{\swarrow\beta}
&
\ar@{} [1,0]|{=}
&
\Set^\opp
\ar[0,1]^-{D^\opp}
\ar[1,0]_{T^\opp}
&
\Pos^\opp
\ar[0,1]^-{P'}
\ar[1,0]^{T'^\opp}
\ar@{} [1,-1]|{\swarrow\alpha^\opp}
&
\DL
\ar[1,0]^{L'}
\ar@{} [1,-1]|{\ \ \ \swarrow\delta'}
\\
\Set^\opp
\ar[0,1]_-{P}
&
\BA
\ar[0,1]_-{W}
&
\DL
&
&
\Set^\opp
\ar[0,1]_-{D^\opp}
&
\Pos^\opp
\ar[0,1]_-{P'}
&
\DL
}
}
\end{equation}
We call $(L',\delta')$ \emph{the maximal positive fragment} 
of $(L,\delta)$ if
$\beta$ is an isomorphism.
\end{definition}

\medskip\noindent Recall that we have defined in Sections~\ref{subsec:BAlogic} and~\ref{subsec:DLlogic}, respectively, the logics $\mathbf L$, $\mathbf L'$ induced by $T$ and (an extension) $T'$ as $\mathbf L = PT^\op S$ and $\mathbf L' = P'T'^\op S'$ on discretely finitely generated free objects. As explained in the introduction, our desired result is to prove that a certain canonically given $\mathbf L' W \to W\mathbf L$, denoted by $\boldsymbol\beta$, is an isomorphism.  
The difficulty, as well as the need for the proviso that $T$ preserves weak pullbacks, stems from the fact that in $\DL$ (as opposed to $\BA$) the class of functors determined on free algebras on finitely many discrete generators is strictly smaller than the class of functors determined on finitely presentable (=finite) algebras, as 
Example~\ref{ex:not strongly finitary} will show.

\medskip\noindent
As stepping stones, therefore, we first investigate what happens in the cases where 
\begin{itemize}
\item the functor $L'$ is determined by $P'T'^\op S'$ on all algebras, 
\item the functor $\bar L'$ is determined by $P'T'^\op S'$ on the full subcategory $\DL_f$ of finitely presentable algebras, 
\end{itemize}
before turning to the situation of the functor $\mathbf L'$ determined by $P'T'^\op S'$ on the full subcategory $\DL_\ff$ of strongly finitely presentable algebras (=free algebras on finitely many discrete generators).

%========================================================%

\subsection{Overview}\label{sec:overview}
To summarize the situation, consider
\begin{equation}\label{eq:P'T'S'}
\vcenter{
\xymatrix@R=20pt{
\DL \ar[rr]^{L',\ \bar L', \ \mathbf L'} && \DL\\
\DL_f \ar[u]^{J'}\ar[urr]|-{\ \bar L'_f\ } & \\
\DL_\ff 
\ar`l[uu]`[uu]^{\mathbf J'}[uu]
\ar[u]^{J''}\ar[uurr]_{\mathbf L'_\ff}
}
}
\end{equation}
where we use subscripts to show the restrictions of $\bar L'$ and $\mathbf L'$ to the corresponding subcategories of $\DL$ where these functors coincide by definition with $P'T'^\op S'$. 

\medskip\noindent
Section~\ref{sec:allalgebras} treats the case of $L'$. Using the dual adjunction of $\DL$ and $\Pos$, one easily obtains the required isomorphism $\beta: L'W \to WL$. But on the other hand, presenting the logic of $L'$ would require operations of infinite arity. 

\medskip\noindent
Section~\ref{finitary} achieves the restriction to finitary syntax, by reducing $L'$ to a functor \linebreak $\bar L':\DL \to \DL$ that is determined by the action $P'T'^\op S'$  on the full subcategory $\DL_f$ on \emph{finitely presentable} distributive lattices. This setting has the advantage that $\DL_f$ is dually equivalent to $\Pos_f$, from which we obtain again the required isomorphism $\bar \beta:\bar L'W \to W \bar L $. On the other hand, such functors $\bar L'$ are in general not presentable by \emph{monotone} operations and equations.

\medskip\noindent
Finally, Section~\ref{sec:DLff} restricts to the category $\DL_\ff$ of \emph{discretely finitely generated free} distributive lattices, which guarantees that the logic of $\mathbf L'$ can be indeed presented by monotone operations and equations, and that the category of $\mathbf L'$-algebras is a(n ordered) variety of modal algebras. In order to obtain that the corresponding $\boldsymbol\beta:\mathbf L 'W \to W \mathbf L$ is an isomorphism, we give conditions under which the functors $\bar L'$ and $\mathbf L'$ coincide. This is done by determining when $\bar L'_f$ in \eqref{eq:P'T'S'} is a left Kan extension of its restriction along $J'':\DL_\ff \to \DL_f$.

%========================================================%

\subsection{The case of $L'=P'T'^\op S'$ on all algebras}\label{sec:allalgebras}

We shall associate to any extension $\alpha:DT\to T'D$ the pairs $(L,\delta)$ and $(L',\delta')$ corresponding to $T$ and $T'$ respectively, with $L=PT^\op S$ and $\delta=PT^\opp\epsilon:PT^\opp SP\to
PT^\opp
$, $L'=P'{T'}^\op S'$ and $\delta'$ being defined analogously. We then immediately obtain an isomorphism $\beta$ by the following:

\begin{proposition}\label{prop:canonical_fragment}
Given an extension $\alpha:DT\to T'D$, the natural isomorphism $\beta:L'W\to WL$ given by the composite below 
\[
\xymatrix{
\BA
\ar[0,1]^-{S}
\ar[1,0]_W
&
\Set^\opp
\ar[1,0]_{D^\opp}
\ar[0,1]^-{T^\opp}
&
\Set^\opp
\ar[1,0]^{D^\opp}
\ar[0,1]^-{P}
&
\BA
\ar@{<-} `u[lll] `[lll]_-{L} [lll]
\ar[1,0]^W
\\
\DL
\ar[0,1]_-{S'}
%\ar[-1,0]^{W}
&
\Pos^\opp
\ar[0,1]_-{T'^\opp}
&
\Pos^\opp
\ar[0,1]_-{P'}
\ar@{} [-1,-1]|{\nearrow\alpha^\opp}
&
\DL
%\ar[-1,0]_{W}
\ar@{<-} `d[lll] `[lll]^-{L'} [lll]
}
\]
exhibits $L'=P'{T'}^\op S'$ as the maximal 
positive fragment of $L=PT^\op S$.
\end{proposition}

\begin{proof}
This follows from $(D^\opp,W)$ being a morphism of adjunctions (see~\eqref{eq:morph_of_adj}). 
\end{proof}

%========================================================%

\subsection{The case of $\bar L'=P'T'^\op S'$ on finitely presentable algebras} 
\label{finitary}

A similar result holds if we define logics via $PT^\op SA$ for finitely
presentable $A$, as we are going to show now. 
To this end, we use the subscript $(-)_f$ to denote the restriction of categories and (domain-codomain) functors to finite 
objects as e.g. when writing as earlier the dense inclusions $I:\Set_f\to\Set$, $I':\Pos_f\to\Pos$, $J:\BA_f\to\BA$ and $J':\DL_f\to\DL$. 

\medskip\noindent Since $\Pos$ is locally finitely presentable as a closed category, and the underlying ordinary categories $\Set_o$, $\DL_o$, $\BA_o$ are also locally finitely presentable, it follows from~\cite{kelly:lfp} 
that the finitely presentable objects in all the above categories are precisely the same as in the ordinary case 
i.e., the ones for which the underlying sets are finite.

\medskip\noindent Note that we have the following commuting diagram
\begin{equation}\label{eq:square_of_transformations}
\xymatrix{ 
S_f\dashv P_f \ar[0,2]^-{(D_f^\op,W_f)}
 \ar[1,0]_{(I^\op,J)} & & S'_f \dashv P'_f \ar[1,0]^{(I'^\op,J')} \\
 S\dashv P \ar[0,2]_-{(D^\op,W)} & & S'\dashv P'
 }
\end{equation} 
in the category of transformations of adjoints.

\medskip\noindent Define now $(\bar L,\bar{\delta})$ for $T$ as 
\[
\bar L=\Lan_J(PT^\op S J) \qquad \mbox{ and } \qquad \bar{\delta}:\bar LP \to PT^{op}
\]
as the adjoint transpose of $\bar L\to  PT^\op S$ arising from the universal property of the left Kan extension $\bar L$. By construction, $\bar L$ is finitary and is given by $PT^\op S$ on finite(ly presentable) Boolean algebras. Similarly, obtain $(\bar L',\bar \delta')$ for $T'$ on distributive lattices.

\medskip\noindent The forgetful functor $W:\BA\to \DL$ is finitary, being a left adjoint (Section~\ref{sec:morph}).
Thus in order to obtain an (iso)morphism $ \bar \beta:\bar L'W \to W\bar L $ 
between finitary functors, it will be enough to provide its restriction along $J$ to finitely presentable objects. But we can get such a transformation from the isomorphism of Proposition~\ref{prop:canonical_fragment}, namely 
\begin{equation}
\label{eq:bar_beta_f}
\bar \beta_f: \bar L' W J \cong \bar L'J'W_f \cong P'{T'}^\op S'J'W_f \cong WPT^\op SJ \cong W\bar L J
\end{equation}
where the second and the last isomorphisms are given by the units of left Kan extensions. 

\begin{remark}\label{rem:fpL=sfpL}
Recall the definition of $\mathbf L$ from~Equation~\eqref{eq:LFn}. Since every finitely presentable non-trivial Boolean algebra is a retract of a finitely generated free algebra, we can identify $\mathbf L= \bar L$, see e.g.~\cite[Proposition~3.4]{kurz-rosicky:strong}. 
\end{remark}

\medskip\noindent According to the above, we then have  

\begin{proposition}\label{prop:canonical_finitary_fragment}
The isomorphism $\bar\beta$ exhibits $\bar L=\Lan_{J'}(P'{T'}^\op S'J')$ as the maximal positive fragment of $(\mathbf L,\mathbf\delta)$.
\end{proposition}

\begin{proof}
The easiest way to check~Equation~\eqref{eq:fragment} is to show that $\bar \beta_f$, defined by~Equation~\eqref{eq:bar_beta_f}, fulfills
\begin{equation*}\label{eq:sf_fragment}
\vcenter{
%\xymatrixrowsep{10pt}
\xymatrix@R=7pt@C=35pt{
&
\Pos_f^\op
\ar[1,1]^{P'_f}
%\ar@{}[2,0]|{\downarrow\pi}
&
\\
\Set_f^\op
\ar[-1,1]^{D_f^\op}
\ar[1,1]^{P_f}
\ar[2,0]_{T^\op I^\op}
&
&
\DL_f
\ar[2,0]^{L'J'}
\\
\ar@{}[1,1]|{\downarrow\delta_f I^\op}
&
\BA_f
\ar[-1,1]^{W_f}
\ar[2,0]|-{L J}
\ar@{}[1,1]|{\downarrow\bar \beta_f}
&
\\
\Set^\op
\ar[1,1]_{P}
&
&
\DL
\\
&
\BA
\ar[-1,1]_{W}
&
}
}
\quad
=
\quad
\vcenter{
%\xymatrixrowsep{1pc}
\xymatrix@R=10pt@C=40pt{
&
\Pos_f^\op
\ar[1,1]^{P'_f}
\ar[2,0]|{T'^\op I'^\op}
&
\\
\Set_f^\op
\ar[-1,1]^{D_f^\op}
\ar[2,0]_{T^\op I^\op}
\ar@{}[1,1]|{\downarrow\alpha^\op I^\op}
&
\ar@{}[1,1]|{\downarrow\delta'_f I'^\op}
&
\DL_f
\ar[2,0]^{L' J'}
\\
&
\Pos^\op
\ar[1,1]_{P'}
%\ar@{}[2,0]|{\downarrow\pi}
&
\\
\Set^\op
\ar[-1,1]_{D^\op}
\ar[1,1]_{P}
&
&
\DL
\\
&
\BA
\ar[-1,1]_{W}
&
}
}
\end{equation*}
But this follows from Proposition~\ref{prop:canonical_fragment}
and~Equation~\eqref{eq:square_of_transformations}.
\end{proof}

\noindent This proposition does not yet give us the desired result, as the endofunctor $\bar L'$ on $\DL$ is not necessarily determined by its action on \emph{discretely finitely generated free algebras} 
and, therefore, is not guaranteed to be equationally definable and to give rise to a variety of modal algebras.

\begin{example}\label{ex:not strongly finitary}
We give an example of a locally monotone functor $\bar L': \DL\to\DL$
that is finitary but does not preserve sifted colimits i.e., it is not determined by its action on discretely finitely generated free algebras.
Consider the composite (comonad) 
$$
\bar L' : \xymatrix{
\DL
\ar[0,1]^-{K}
&
\BA
\ar[0,1]^-{W}
&
\DL
}
$$
where $K$ is the right adjoint of the forgetful functor $W$ (recall that we met $K$  in Remark~\ref{rmk:functor-K} assigning to each distributive lattice $A$ the Boolean subalgebra $KA$ of its complemented elements). We shall exhibit below a reflexive coinserter that it is not preserved by $\bar L '$. Since reflexive coinserters are sifted colimits, this means that $\bar L':\DL\to\DL$ does not preserve (all) sifted colimits.

Let $e:D2\to X$ be the obvious embedding of the two-element
discrete poset into the poset $X$ with three elements
$0$, $1$ and $\top$, satisfying $0<\top$ and $1<\top$. 
Then $e$ is a coreflexive inserter in $\Pos$ of the 
pair $f,g:X\to Y$, where $Y$ is the poset
$$
\xy
\POS (000,010) *+{\top_1} = "t1"
   ,  (010,010) *+{\top_2} = "t2"
   ,  (000,000) *+{0} = "0"
   ,  (010,000) *+{1} = "1"
%%%
\POS "0" \ar@{-} "t1"
\POS "1" \ar@{-} "t1"     
\POS "0" \ar@{-} "t2"
\POS "1" \ar@{-} "t2" 
\endxy
$$
and $f(0)=g(0)=0$, $f(1)=g(1)=1$, $f(\top)=\top_1$,
$g(\top)=\top_2$. 

By applying $P':\Pos^\op\to\DL$ to the above coreflexive
inserter in $\Pos$ we obtain
a reflexive coinserter in $\DL$, which can either be checked directly or by appealing to Proposition~\ref{P'corefl eq} below.
Observe that $P'e:P'D2\to P'X$ is the embedding of the four-element
Boolean algebra into the distributive lattice
$$
\xy
\POS (000,000) *+{\{0,\top\}} = "l"
   , (020,000) *+{\{1,\top\}} = "r"
   , (010,-10) *+{\{\top\}} = "b"
   , (010,010) *+{\{0,1,\top\}} = "t"
   , (010,-20) *+{\emptyset} = "bb"
%%%
\POS "bb" \ar@{-} "b"
\POS "b" \ar@{-} "l"     
\POS "b" \ar@{-} "r"
\POS "l" \ar@{-} "t" 
\POS "r" \ar@{-} "t"
\endxy
$$
Then it is easy to see that $\bar L'P'e: \bar L'P'X\to \bar L'P'D2$ fails to be surjective
(since $\bar L'P'D2$ has four elements and $\bar L'P'X$ only two elements). Hence
$\bar L'P'e$ is not the coinserter in $\DL$ of the parallel pair $\bar L'P'f$, $\bar L'P'g$.

Although the locally monotone functor $\bar L'=WK$ fails to preserve reflexive coinserters, it is however finitary, being a composite of such. Indeed, $W$ is left adjoint, while for $K$ notice the following: $K$ is a right adjoint functor between locally finitely presentable categories, whose left adjoint $W$ preserves finitely presentable objects (see paragraph (2) in the proof of \cite[Theorem~1.66]{adamekrosicky}). %
%More in detail: consider a filtered colimit $A=\mathsf{colim}A_i$ in $\DL$; then we have
%\begin{align*}
%&\BA(B, K\mathsf{colim} A_i) &\cong & \, \DL(WB, \mathsf{colim} A_i) &\cong & \, \mathsf{colim} \, \DL(WB, A_i) \\
%& &\cong & \, \mathsf{colim} \, \BA(B, KA_i) &\cong & \, \BA(B,\mathsf{colim}KA_i)
%\end{align*}
%for every $B$ a finite(ly presentable) Boolean algebra. Use now that each Boolean algebra is a filtered colimit of such, and $W$ preserves filtered colimits being left adjoint, to obtain that 
%\begin{align*}
%& \BA(B,K\mathsf{colim} A_i) &\cong & \, \BA(\mathsf{colim} B_j, K\mathsf{colim} A_i) 
%& \cong & \, \mathsf{lim} \, \BA(B_j, K\mathsf{colim} A_i) \\
%&& \cong & \, \mathsf{lim} \, \BA(B_j, \mathsf{colim} \, KA_i) 
%& \cong & \, \BA(\mathsf{colim} B_j, \mathsf{colim} \, KA_i) \\
%&& \cong & \, \BA(B, \mathsf{colim} KA_i)
%\end{align*}
%for each Boolean algebra $B$. Now the result follows from the Yoneda lemma. 
%\qed
\end{example}

\medskip\noindent The next paragraph will investigate when the functor given by $\bar L'A=P'T'^\op S'A$ on finitely presentable $A$ is not only finitary but also  preserves sifted colimits.

%========================================================%

\subsection{The case of $\mathbf L' =P'T'^\op S'$ on discretely finitely generated free algebras} \label{sec:DLff}

Recall that in Section~\ref{sec:sifted} we have denoted by $\mathbf J:\BA_\ff\to\BA$ and $\mathbf J':\DL_\ff\to\DL $ the inclusion functors of the full subcategories spanned by the algebras which are free on finite discrete posets.

\begin{definition} 

Let $T'$ be an endofunctor of $\Pos$. We define the logic for $T'$ to be the pair $(\mathbf L',\boldsymbol \delta')$, where: 
\begin{enumerate}
\item 
$\mathbf L':\DL\to\DL$ is an endofunctor of $\Pos$ preserving sifted colimits, whose restriction to discretely finitely generated free distributive lattices is $P'T'^\op S'\mathbf J'$, that is, $\mathbf L'=\Lan_{\mathbf J'}(P'{T'}^{\op}S'\mathbf J')$. 
\[
\xymatrix{\DL_\ff \ar@{}[1,3]|{\nearrow} \ar[0,3]^{\mathbf J'} \ar[1,0]_{\mathbf J'} & & & \DL \ar[d]^{\mathbf L'} \\
\DL \ar[0,1]_-{P'} & \Pos^\opp \ar[0,1]_-{{T'}^\op} & \Pos^\op \ar[r]_{S'} & \DL %\ar@{<-} `d[lll] `[lll]^-{L'} [lll] 
 } 
\]
\item 
$\boldsymbol \delta':\mathbf L'P'\to P'{T'}^{\op}$ is the pasting composite 
\begin{equation*} \label{eq:canonical_strongly_finitary_logics}
\xymatrix{ 
\DL \ar[d]_{\mathbf L'} \ar[r]^{S'} & \Pos^\op \ar[r]^{{T'}^\op} & \Pos^\op \ar[d]_{P'} 
\ar@<+1.25pt>@{-}`r[dr][dr] 
\ar@<-1.25pt>@{-}`r[dr][dr]
& \\
\DL \ar@{}[urr]|{\nearrow} \ar@{=}[rr] && \DL \ar[r]_{S'} \ar@{}[ur]|{\nearrow\eps'} & \Pos^\op
} \end{equation*}
that is, the adjoint transpose of the natural transformation $\mathbf L'\to P'{T'}^{\op}S'$, which in turn is given by the universal property of the left Kan extension $\mathbf L'$.
\end{enumerate}

\end{definition}

\begin{remark} 
By Corollary~\ref{cor:when L is lan}, $\mathbf L'$ defined above preserves sifted colimits. Thus, by Theorem~\ref{thm:eq-pres}, the functor $\mathbf L'$ admits an equational presentation by monotone operations, which subsequently gives rise to a positive modal logic concretely given in terms of modal operators and axioms. 
\end{remark}

\noindent Recall from the previous section that we have introduced the functor $\bar L'$, which on finitely presentable distributive lattices is $P'T'^\op S'$, while now we have defined ${\bf L'} = P'T'^\op S'$ only on discretely finitely generated free distributive lattices. The next theorem will provide sufficient conditions for these two functors $\bar L'$ and $\mathbf L'$ to coincide. 

\medskip\noindent First note that the restrictions of $\bar L'$ and $\mathbf L'$ to $\DL_\ff$ coincide by definition with $P'T'^\op S'\mathbf J'$. Abbreviate as in \eqref{eq:P'T'S'} $\bar L'_f = \bar L' J'$ and $\mathbf L'_\ff=\mathbf L'\mathbf J'$. Also observe that due to \cite[\mbox{Theorem~4.47}]{kelly:enriched}, the functor $\mathbf L'=\Lan_{\mathbf J'}\mathbf L'_\ff$ can equally be expressed as the iterated Kan extension 
$$
\Lan_{J'} (\Lan_{J''} \mathbf L'_\ff)
$$
Because $\bar L'=\Lan_{J'}\bar L'_f$, it will be then enough to show that $\bar L'_f=\Lan_{J''} \mathbf L'_\ff$, or, in other words, that $P'T'^\op S'J'$ is the left Kan extension of its restriction to $J''$. 
\begin{equation}\label{eq:P'T'S'2}
\vcenter{
\xymatrix@R=20pt{
\DL \ar[rr]^{\bar L' \ , \  \mathbf L'} && \DL\\
\DL_f \ar[u]^{J'}\ar[urr]|-{\ \bar L'_f\ } & \\
\DL_\ff \ar`l[uu]`[uu]^{\mathbf J'}[uu]
\ar[u]^{J''}
\ar[uurr]_{\mathbf L'_\ff}
}
}
\end{equation}
This will follow once we have shown that $J''$ is dense and that $P'T'^\op S'J'$ preserves the colimits of its density presentation.

\begin{theorem}\label{thm:DLffisdense}
The inclusion $J'':\DL_\ff\to\DL_f$ from discretely finitely generated free distributive lattices to finitely presentable distributive lattices is dense. A density presentation is given by reflexive coinserters.
\end{theorem}

\begin{proof}
According to the definition of density presentation \cite[Theorem~5.19]{kelly:enriched}, we have to show that reflexive coinserters exist in $\DL_f$, that they are $J''$-absolute, and that every object in $\DL_f$ can be constructed from objects in $\DL_\ff$ using reflexive coinserters.

First, note that $\DL_f$ has reflexive coinserters since a distributive lattice is in $\DL_f$ if and only if it is finite, and since a coinserter of finite distributive lattices is finite.

Second, to say that a reflexive coinserter is $J''$-absolute is to say that it is preserved by $\DL(F'Dn, -)$ for all finite sets $n$. But $\DL(F'Dn, -)\cong \Pos(Dn,U'-)$, which preserves reflexive coinserters since finite products preserve reflexive coinserters, and so does $U'$.

Finally, we have to prove that every finite distributive lattice is in the closure of $\DL_\ff$ under reflexive coinserters. Notice that every distributive lattice $F'X$ which is free over a finite poset $X$ is  such a reflexive coinserter. This is immediate from $F'$ preserving colimits and from Proposition~\ref{poset=refl_coins_discrete} which presents any poset $X$ as a reflexive coinserter of discrete posets. 
It remains to show that every finite distributive lattice $A$ is a reflexive coinserter of finitely generated free ones. 
Let  $A$ be a finite distributive lattice and consider the counit 
$$
\eps_A : F'U'A\to A
$$ 
of the $\Pos$-enriched adjunction $F'\dashv U':\DL\to\Pos$. Since $\eps_A$ is surjective, it is a coinserter of some pair 
$$
\xymatrix{
A' 
\ar@<.5ex>[r]^-{l'} 
\ar@<-.5ex>[r]_-{r'} 
& 
F'U'A
}
$$ 
(by factoring the pair through its image, we can assume without loss of generality that $A'$ is finite). Now pre-compose this pair with $\eps_{A'} : F'U'A' \to A'$ to obtain 
\begin{equation}\label{eq:coins free finite posets}
\xymatrix@C=35pt{
F'U'A' 
\ar@<.5ex>[r]^-{l'\circ \eps_{A'}} 
\ar@<-.5ex>[r]_-{r'\circ\eps_{A'}}
& 
F'U'A 
\ar[r]^-{\eps_A} 
& 
A
}
\end{equation}
which is again a coinserter of free distributive lattices on \emph{finite} posets since $\eps_{A'}$ is surjective. 
Notice that we can always turn the coinserter~\eqref{eq:coins free finite posets} into a \emph{reflexive} one, namely
\begin{equation}\label{eq:refl coins free finite posets}
\xymatrix@C=35pt{
F'U'A+F'U'A' 
\ar@<.5ex>[r]^-{l} 
\ar@<-.5ex>[r]_-{r}
& 
F'U'A 
\ar[r]^-{\eps_A} 
& 
A
}
\end{equation}
where the parallel arrows $l=[\id_{F'U'A}, l'\circ\eps_{A'}]$ and $r=[\id_{F'U'A}, r'\circ\eps_{A'}]$ are given by the universal property of the coproduct, with common splitting $\mathsf{in}:F'U'A \to F'U'A +F'U'A'$ provided by the canonical injection into the coproduct, as the diagram below indicates:
$$
\xymatrix@C=50pt{
F'U'A' 
\ar[r]^-{\eps_{A'}} 
\ar[d]_{\mathsf{in'}}
&
A' 
\ar@<-0.4ex>[d]_{r'}
\ar@<0.4ex>[d]^{l'}
&
\\
F'U'A+F'U'A' 
\ar@<-0.35ex>[r]_-r
\ar@<+0.35ex>[r]^-l
&
F'U'A
\ar[r]^-{\eps_{A}}
&
A
\\
F'U'A 
\ar[u]^{\mathsf{in}}
\ar@<-1.7pt>@{->}`r[ur]_-\id[ur]
\ar@<+1.7pt>@{->}`r[ur]^-\id[ur]
}
$$

\noindent
To see that~\eqref{eq:refl coins free finite posets} is indeed a coinserter, notice first that the inequality $\eps_A\circ l \le \eps_A\circ r$ follows from $\eps_A\circ l' \le \eps_A\circ r'$ and from the 2-dimensional property of the $\Pos$-enriched coproduct.
 
Next, let $h:F'U'A \to B$ be an arrow with $h\circ l\le h\circ r$. It follows that
$$
h \circ l'\circ \eps_{A'} = h\circ l\circ \mathsf{in}\le h\circ r\circ\mathsf{in} = h\circ r'\circ \eps_{A'}
$$
and since $\eps_{A'}$ is onto, we obtain 
$$
h\circ l'\le h\circ r'
$$
But $\eps_A$ is the coinserter of the pair $l'$, $r'$, therefore there is an arrow $k:A \to B$ such that $h=k\circ\eps_A$. Since $\eps_{A}$ is onto, $k$ is unique. We have therefore shown that in~\eqref{eq:refl coins free finite posets}, the morphism $\eps_A$ is the reflexive coinserter of the pair $l$, $r$.

Using now in~\eqref{eq:refl coins free finite posets} that 
$$
F'U'A+F'U'A' \cong F'(U'A+U'A')
$$ 
we see that the finite distributive lattice $A$ can be obtained as a \emph{reflexive} coinserter of distributive lattices that are \emph{free} over finite posets. Since we have already established that the latter are in the closure of $\DL_\ff$ under reflexive coinserters, the proof is finished.
\end{proof}

\begin{remark}
Since the class of reflexive coinserters is definable by a weight, the theorem also shows that $\DL_f$ is the free cocompletion of $\DL_\ff$ by reflexive coinserters, see Proposition~\cite[Proposition~4.1]{kelly-schmidt}. Note that the only particular property of the ordered variety $\DL$ used in the proof is that $\DL$ is locally finite, that is, that finitely generated algebras are finite. 
\end{remark}

\medskip\noindent It follows from the theorem that a functor with domain $\DL_f$ is the left Kan extension of its restriction along $J'':\DL_\ff \to \DL_f$ if it preserves reflexive coinserters. Being interested in the composite functor $P'T'^\op S'J':\DL_f \to \DL$, our next step is to establish that $P':\Pos^\op\to\DL$ preserves  reflexive coinserters. We split this in several lemmas.

\begin{lemma}\label{exists}
Let $e:E\to X$ be an embedding (i.e., a monotone and order-reflecting map) of posets. Then $[e,\Two]$ has a right inverse.
\end{lemma}

\begin{proof}
Remember that any poset can be seen as a category enriched over the two-elements poset $\Two$, and any monotone map $e:E\to X$ as a $\Two$-enriched functor. 

Pre-composition with $e$ gives a monotone map (hence, a $\Two$-functor) $[e,\Two]:[X,\Two]\to [E, \Two]$ which always has a left adjoint $\exists_e:[E,\Two]\to[X, \Two]$, given by \emph{left} Kan extension along $e$ (but also a right adjoint $\forall_e:[E,\Two] \to [X, \Two]$, provided by the \emph{right} Kan extension along $e$). Explicitly, for a monotone map $f:E \to \Two$, the left adjoint acts as follows: $\exists_e(f):X \to \Two$ is the monotone map given by 
\[
\exists_e (f) (x) = 1 \quad \Longleftrightarrow \quad \exists a\in E \, . \, (f(a)= 1 \, \land \, e(a) \leq x)
\]
for any $x \in X$. 

The unit and counit of the adjunction $\exists_e \dashv [e,\Two]$ are the inequalities 
\begin{eqnarray} 
\id_{[E,\Two]} \leq [e,\Two] \circ \exists_e \label{eq:unit-adj-exists} 
\\
\exists_e \circ [e, \Two] \leq \id_{[X,\Two]} \label{eq:counit-adj-exists}
\end{eqnarray}
By hypothesis, $e$ is an embedding. This means precisely that $e$ is fully faithful as a $\Two$-enriched functor. Therefore, the unit $\,f\le [e,\Two]\circ \exists_e(f)\,$ of the left Kan extension $\exists_e(f)$ of $f:E\to\Two$ along $e$ is an isomorphism~\cite[Proposition~4.23]{kelly:enriched}. But since $\Two$ is a poset, isomorphism means equality and we obtain 
\begin{equation}\label{eq:unit-adj-equals}
\id_{[E,\Two]}=[e,\Two]\circ \exists_e. 
\end{equation}
improving~\eqref{eq:unit-adj-exists}.
\end{proof}

\medskip\noindent Remember the notion of an exact square from Definition~\ref{def:exact}. We shall give now an equivalent formulation:

\begin{lemma}[Beck-Chevalley Property]\label{BC}
The diagram~\eqref{exact-sq} exhibits an exact square if and only if $[g,\Two]\circ \exists_f=\exists_\alpha\circ [\beta, \Two]$. That is
\[
\vcenter{
\xymatrix{
E \ar[r]^{\alpha} \ar[d]_{\beta} \ar@{}[1,1]|{\swarrow}& X\ar[d]^f \\
Y\ar[r]_g & Z
}
} 
\qquad
\mbox{is exact}
\qquad  
\Longleftrightarrow
\qquad  
\vcenter{
\xymatrix{
%\ar@{}[dr]|{=}
[E,\Two] \ar[d]_{\exists_\beta}& [X,\Two] \ar[d]^{\exists _f} \ar[l]_{[\alpha,\Two]}  \\
[Y,\Two]  & [Z,\Two]\ar[l]^{[g,\Two]} 
}
}
\qquad
\mbox{commutes.}
\]
\end{lemma}

\begin{proof}It follows easily by direct computation.
\end{proof}

\medskip\noindent As in the proof of subsequent Proposition~\ref{P'corefl eq} we shall encounter split coinserters -- the ordered analogue of split coequalizers~\cite[\textsection~VI.6]{maclane:cwm}, we provide below the precise definition:

\begin{definition}
In a $\Pos$-category, the diagram below is called a \emph{split coinserter} if it satisfies the following equations and inequations:
\begin{equation}\label{eq:split-coins}
\vcenter{
\xymatrix@C=40pt{
A
\ar@<0.4ex>[r]^{d^0} 
\ar@<-0.4ex>[r]_{d^1} 
\ar@{<-}`u[r]`[r]^{t}
&
B \ar[r]^-{c}
\ar@{<-}`d[r]`[r]_{s}
&
C
}
}
\qquad \qquad \qquad 
\begin{cases}
c \circ d^0 \le c \circ d^1
\\
c \circ s=\id_{C}
\\
d^0 \circ t=\id_{B}
\\
d^1 \circ t = s \circ c \le \id_{B}
\end{cases}
\end{equation}
\end{definition}

\begin{proposition}\label{prop:split-coinserter}
\mbox{}\hfill
\begin{enumerate}
\item A split coinserter is a coinserter.
\item Split coinserters are absolute: they are preserved by all \emph{locally monotone} functors.
\end{enumerate}
\end{proposition}

\begin{proof}
The second statement is immediate. To show the first one, let $h:B\to D$ be an arrow such that $h\circ d^0 \le h\circ d^1$. Define $k:C\to D$ by $k=h\circ s$. We have 
$$
k\circ c = h \circ s \circ c \le h
$$ 
and 
$$
k\circ c = h \circ s \circ c = h \circ d^1 \circ t \ge h \circ d^0 \circ t =  h
$$
Therefore $k\circ c = h$. And $k$ is unique with this property since $c$ is (split) epi.
\end{proof}

\medskip\noindent We have now all ingredients to come back to the question of $P'$ preserving reflexive coinserters.

\begin{proposition}\label{P'corefl eq}
The functor $P':\Pos^\op\to \DL$ preserves reflexive coinserters. 
\end{proposition}

\begin{proof}
Notice that $U'P'=[-,\Two]$ and that $U'$ is monadic. Since $\DL$ is an ordered variety, it has all (sifted) colimits, in particular reflexive coinserters, and $U'$ creates them. Thus it is enough to show that $[-,\Two]:\Pos^\op \to \Pos$ preserves reflexive coinserters. 

But reflexive coinserters in $\Pos^\op$ are coreflexive inserters in $\Pos$; consider therefore two monotone maps with common left inverse in $\Pos$
\begin{equation}\label{descent}
\xymatrix{
X 
\ar@<0.4ex>[r]^{f} 
\ar@<-0.4ex>[r]_{g} 
\ar@{<-}`u[r]`[r]^i
&
Y 
}
\end{equation}
The inserter of the above data is realized as the poset $\mathsf{ins}(f,g)=\{x\in X \mid f(x)\leq g(x)\}$ with the order inherited from $X$, together with the inclusion map $\mathsf{ins}(f,g)\overset{e}{\to} X$. In particular, the diagram below is an exact square: 
\[
\xymatrix{\mathsf{ins}(f,g) \ar[r]^-e \ar[d]_e \ar@{}[1,1]|{\swarrow} & X\ar[d]^{f} \\ X \ar[r]_-{g} & Y}
\]
By Lemma \ref{BC}, we obtain $[g,\Two]\circ \exists_{f}=\exists_e\circ [e,\Two]$. Additionally, as both $e$ and $f$ are embeddings (recall that the inserter pair $f,g$ was assumed to have a common left inverse), $[e,\Two]\circ \exists_e=\id_{[\mathsf{ins}(f,g), \Two]}$ and $[f,\Two]\circ \exists_{f}=\id_{[X,\Two]}$ by Lemma \ref{exists}. 

That is, by applying $[-,\Two]$ to the diagram~\eqref{descent} augmented by $\mathsf{ins}(f,g)\overset{e}{\to} X$ and $\exists_e,\exists_{f}$, 
we obtain the split coinserter
\begin{equation}\label{eq:split-refl-coins}
\vcenter{
\xymatrix@C=40pt{
[Y,\Two]
\ar@<0.4ex>[r]^{[f,\Two]} 
\ar@<-0.4ex>[r]_{[g,\Two]} 
\ar@{<-}`u[r]`[r]^{\exists_f}
&
[X,\Two] \ar[r]^-{[e,\Two]}
\ar@{<-}`d[r]`[r]_{\exists_e}
&
[\mathsf{ins}(f,g),\Two]
}
}
\quad \ \
\begin{cases}
[e,\Two]\circ [f,\Two] \le [e,\Two]\circ [g,\Two]
\\
[e,\Two]\circ \exists_e=\id_{[\mathsf{ins}(f,g),\Two]}
\\
[f,\Two]\circ \exists_{f}=\id_{[X,\Two]}
\\
[g,\Two]\circ \exists_{f}=\exists_e\circ [e,\Two] \le \id_{[X,\Two]}
\end{cases}
\end{equation}

\medskip\noindent
where the first inequation is due to $[-,\Two]$ being a locally monotone functor, the next two equations correspond to~\eqref{eq:unit-adj-equals}, the last equation follows from Lemma~\ref{BC}, while the last inequality is due to~\eqref{eq:counit-adj-exists}. Use then Proposition~\ref{prop:split-coinserter} to conclude that $[-,\Two]$ maps coreflexive inserters to (split and reflexive) coinserters. 
\end{proof}

\medskip\noindent The following theorem is the technical result around which this section revolves.

\begin{theorem}\label{thm:strong}
Let $\bar L=\Lan_{J'}(P'T'^\op S'J')$ and $\mathbf L=\Lan_{\mathbf J'}(P'T'^\op S'\mathbf J')$. If $T'$ sends coreflexive inserters of finite posets into coreflexive inserters, then $\bar L'\cong \mathbf L'$.
\end{theorem}

\begin{proof}
Taking up the reasoning preceding~\eqref{eq:P'T'S'2}, we need to show that $P'T'^\op S'J'$ is the left Kan extension of its restriction along $J''$.  We know that $P'T'^\op S'J'$ preserves reflexive coinserters of finite distributive lattices. Indeed, $J'$ is a completion by filtered colimits and preserves all finite colimits~\cite{adamekrosicky}; $S'$ as a left adjoint preserves all colimits (hence it maps reflexive coinserters to coreflexive inserters) and maps finite distributive lattices to finite posets; $T'^\op$ preserves coreflexive inserters of finite posets by hypothesis, and $P'$ sends them to reflexive coinserters, due to Proposition~\ref{P'corefl eq}. Our claim now follows from Theorem~\ref{thm:DLffisdense} and  \cite[Theorem~5.29]{kelly:enriched}.
\end{proof}

\medskip\noindent As a result about modal logics, the theorem can be reformulated as follows.

\begin{corollary}\label{cor:strong}
Let $T$ be an endofunctor on $\Set$ and $T'$ a $\Pos$-extension of $T$ which \emph{preserves coreflexive inserters}. Then $(\bar L', \bar{\delta}')$ and $(\mathbf L', \boldsymbol \delta ')$ coincide. In particular, it follows from Proposition~\ref{prop:canonical_finitary_fragment} that $\mathbf L'$ is \emph{the maximal positive fragment} of $\mathbf L$. 
\end{corollary}

\medskip\noindent The next example shows what can go wrong in case that $T'$ does not preserve coreflexive inserters.

\begin{example}
For $T=\Id$, the corresponding finitary logics is $\mathbf L=\Id$ on $\BA$, with trivial semantics $\boldsymbol \delta:LP\to PT^\op$. It was noticed in Remark~\ref{rem:posetific}\eqref{DC}
that the identity functor also admits as extension the discrete connected components functor $T'=D C$. But the latter preserves neither embeddings, nor coreflexive inserters. The corresponding logic $\mathbf L'$ for $T'$ is given by the constant functor to the distributive lattice $\Two$. Thus the natural transformation $\beta: \mathbf L'W\to W\mathbf L$ from Definition~\ref{def:pos-frag} fails to be an isomorphism (it is just the unique morphism from the initial object).

Whereas the associated `strongly finitary' logic $\mathbf L'$ of $T'=DC$ is just the logic of the constant functor $\Two$ (i.e., plain positive propositional logic), the associated `finitary logic' $\bar L'$ is given by the functor $WK$ on $\DL$. This can be seen as follows:  on finite distributive lattices, $\bar L'$ is $P'D^\op C^\op S' \cong WPC^\op S' \cong W KP' S'\cong WK$. As $WK$ is finitary (Example~\ref{ex:not strongly finitary}), it coincides with $\bar L'$ on all distributive lattices, not just on the finite ones.  \qed
\end{example}

%========================================================%

\medskip\noindent
The next lemma shows that for a locally monotone functor on $\Pos$, preservation of exact squares entails the condition needed in Theorem~\ref{thm:strong}, namely the 
preservation of coreflexive inserters:

\begin{lemma}%\label{ex=>emb}
If $T'$ is a locally monotone endofunctor on $\Pos$ which preserves exact squares, then it preserves embeddings and coreflexive inserters.
\end{lemma}

\begin{proof}
The first assertion follows from the observation~\cite{Guitart80} that each embedding $e:X\to Y$ can be realized as an exact square, namely
\[
\xymatrix{X\ar[r]^{\id} \ar[d]_{\id} \ar@{}[1,1]|{\swarrow}& X\ar[d]^e \\
X\ar[r]_-e & Y}
\]
For the second one, let
\[
\xymatrix{\mathsf{ins}(f,g) \ar[r]^-{e} & X \ar@<0.5ex>[r]^f \ar@<-0.5ex>[r]_g \ar@{<-}`u[r]`[r]^{i} & Y 
}
\]
be a coreflexive inserter. In particular, 
$$\xymatrix{\mathsf{ins}(f,g) \ar@{} [1,1]|{\swarrow} \ar[r]^-e \ar[d]_e & X \ar[d]^f \\
X\ar[r]_-g & Y}$$
is an exact square as remarked in the proof of Proposition~\ref{P'corefl eq}, thus $T'$ maps it to the exact square 
\[
\xymatrix@C=35pt{
T'\mathsf{ins}(f,g) \ar@{} [1,1]|{\swarrow} \ar[r]^-{T'e} \ar[d]_{T'e} & T'X \ar[d]^{T'f} \\
T'X \ar[r]_-{T'g} & T'Y}
\]
Let now $u:U\to T'X$ a monotone map such that $T'f\circ u\leq T'g\circ u$. 
For each $x\in U$, $T'f ( u (x) ) \leq T'g ( u(x) )$, thus there is some $w\in T'\mathsf{ins}(f,g)$ with $u(x)\leq T'e(w)$ and $T'e(w)\leq u(x)$, that is, $u(x)=T'e(w)$. 
As $T'e$ is an embedding, such element $w$ is uniquely determined. Moreover, the assignment $x\mapsto w$ is monotone, as if $x_1\leq x_2$, then $T'e(w_1)=u(x_1)\leq u(x_2)=T'e(w_2)$ and $T'e$ is again an embedding as shown earlier, hence $w_1\leq w_2$. This covers the 1-dimensional aspect of inserters. For the rest, use one more time that $T'e$ is an embedding.
\end{proof}

%========================================================%

\medskip\noindent As a consequence of all the results of this section and of Theorem~\ref{thm:wpb_exsq}, we obtain our main theorem on positive modal logic.

\begin{theorem}\label{mainthm}
Let $T:\Set\to\Set$ be a weak-pullback preserving functor and $T':\Pos\to\Pos$ its posetification. 
Let $(\mathbf L,\boldsymbol\delta)$ and $(\mathbf L',\boldsymbol \delta')$ be the associated logics of $T$ and $T'$, that is $\mathbf L=\Lan_{\mathbf J}(PT^\op S\mathbf J)$ and $\mathbf L'=\Lan_{\mathbf J'}(P'T'^\op S'\mathbf J')$. 
Then $(\mathbf L',\boldsymbol \delta')$ is the maximal 
positive fragment of $(\mathbf L,\boldsymbol\delta)$.
\end{theorem}

\begin{corollary}\label{cor:mainthm}
Under the hypotheses of Theorem~\ref{mainthm}, the 
adjunction $W\dashv K:\DL\to \BA$ lifts to an adjunction 
$\widetilde W \vdash \widetilde K:\Alg(\mathbf L') \to \Alg(\mathbf L)$.
\end{corollary}

\begin{proof}
Recall from Definition~\ref{def:pos-frag} that the statement of  Theorem~\ref{mainthm} implies that there is a natural isomorphism 
$$\boldsymbol \beta: \mathbf L' W \to W\mathbf L.$$ 
By Remark~\ref{rmk:functor-K}, and by standard doctrinal adjunction~\cite{kelly:doctrinal}, we know then that the adjunction $W\dashv K:\DL\to\BA$ lifts to an adjunction between the corresponding categories 
of algebras, where the lifting $\widetilde W$ of $W$ maps an $\mathbf L$-algebra 
$$
\xymatrix{
\mathbf L B
\ar[r]^{b}
& 
B
}
$$
to the $\mathbf L'$-algebra  
$$
\xymatrix{
\mathbf L'WB 
\ar[r]^{\boldsymbol\beta} 
& 
W\mathbf LB 
\ar[r]^{Wb} 
& 
WB
}
$$
\end{proof}

\medskip\noindent The above corollary shows in particular that $\Alg(\mathbf L)$ is a full coreflective subcategory of 
$\Alg(\mathbf L')$ (using that $\tilde W$ is fully-faithful, property inherited from $W$), and that $\widetilde W$ preserves initial algebras. In other words, the Lindenbaum algebra for $\mathbf L$ is the same as the one for $\mathbf L'$ (the analogous statement for free algebras only holds if the generating atomic propositions are closed under complement).

\begin{remark}
Our introductory example of positive modal logic is now regained as an instance of this theorem. It can also easily be adapted to Kripke polynomial functors. More interesting are the cases of probability distribution functor and of multiset functor. We know from the theorem above that they have maximal positive fragments, but their explicit description still needs to be worked out. 
\end{remark}

\medskip\noindent
To see how Dunn's completeness result can be obtained in our setting, we first remark that, from an algebraic point of view, completeness follows from Theorem~\ref{mainthm}. Indeed, let $I,I'$ be the initial algebras (i.e., the  Lindenbaum algebras) for the functors $\mathbf L$ and $\mathbf L'$, respectively. Then the induced arrow $I'\to WI$ is an isomorphism. Hence, if two elements (i.e., formulas) are equal in $I$ they also must be equal in $I'$, which implies, by completeness of equational logic, that every proof with $\mathbf L$-axioms can be imitated with $\mathbf L'$-axioms. Here, the terminology of axioms refers to presentations of the functors in the sense of Definitions~\ref{def:presentation} and \ref{def:presentation2}.

\medskip\noindent 
From the point of view of Kripke semantics, completeness means that if two formulas have the same semantics, then they are equal. In other words, recalling Section~\ref{sec:Completeness}, with $(X,\xi)$ ranging over all coalgebras, the family $\sem{\cdot}_{(X,\xi)}:I\to PX$, or rather $\sem{\cdot}_{(X,\xi)}:I'\to P'X$ must be jointly injective. As  in \cite{kupkeetal:algsem}, this follows from the one-step semantics $\delta':\mathbf L'P'\to P'T'$ being injective. We sketch below (Paragraphs~\hyperref[par:a]{\textbf A}-\hyperref[par:d]{\textbf D}) some of the details.

\medskip\noindent 

\paragraph{\textbf A}\label{par:a}%
First, we show that for all posetifications $T':\Pos\to\Pos$ and $X'$ in
$\Pos$ the one-step semantics $\delta'_{X'}:\mathbf L'P'X'\to P'T'X'$ is injective. Up to replacing $\delta, L,P,T$ by $\delta',\mathbf L', P', T'$, the proof is verbatim the same as in \cite[Lemma 6.12]{kurz-petr:manysorted}, using Theorem~\ref{thm:strong} in order to know that $\mathbf L'P'X'$ is given as a filtered colimit as in \cite[(44)]{kurz-petr:manysorted} and using Proposition~\ref{prop:T'-pres-monot-surj} in order to know that the posetification $T'$ preserves surjections.

\medskip\noindent

\paragraph{\textbf B}\label{par:b}%
Next, in analogy with~\cite[Lemma~6.14]{kurz-petr:manysorted}, we show that $\mathbf L'$ preserves injections. Let us emphasize that whereas all functors $\BA\to\BA$ with a presentation by operations and equations preserve injections, the same is not true for functors $\DL\to\DL$, see~\cite[Section~5.3]{myers:phd}.

\begin{proposition}\label{prop:Lpresinj}
Assume that the hypotheses of Theorem~\ref{mainthm} hold. Then $\mathbf L'$ preserves monomorphisms of distributive lattices. 
\end{proposition}

\begin{proof}
We shall proceed in two steps:
\begin{enumerate} 
\item 
$\mathbf L'$ preserves monomorphisms between finite distributive lattices. This can be seen as follows: recall that on finite distributive lattices, $\mathbf L'$ is $P' {T'}^\op S'$ (Theorem~\ref{thm:strong}), and let $m : A\to B$ be an injective morphism between two finite distributive lattices. According to~\cite[Theorem~4.2]{zawadowski}, $m$ is the coreflexive inserter of its co-comma square:
\[
\xymatrix{
A
\ar[r]^m
&
B
\ar@<+0.3ex>[r]^p \ar@<-0.3ex>[r]_q
&
C
}
\]
Notice that $C$ is also finite. But on finite distributive lattices, $S'$ is an equivalence. Therefore $S'm$ is the reflexive coinserter of its comma object $(S'p,S'q)$, in particular, a monotone surjection. By Proposition~\ref{prop:T'-pres-monot-surj}, $T'S'm$ will also be a surjective monotone map between posets, thus a monomorphism in $\Pos^\op$. As $P'$ is right adjoint, it sends ${T'}^\op S'm$ to a monomorphism in $\DL$. Thus $\mathbf L'm \cong P'{T'}^\op S'm$ is again a monomorphism.

\item $\mathbf L'$ preserves arbitrary monomorphisms. Let $m : A \to B$ be a monomorphism of distributive lattices. But $\DL$ is a locally finitely presentable category, the finitely presentable objects being precisely the finite ones. Thus~\cite[Corollary~4.3]{bird} applies to conclude that $m$ can be expressed as a filtered colimit of monomorphisms between finite distributive lattices. We can now use step (1) and that $\mathbf L'$ preserves filtered colimits, to obtain that $\mathbf L' m$ is again a monomorphism (see also~\cite[Corollary 1.60]{adamekrosicky}). 
\qedhere
\end{enumerate}
\end{proof}

\begin{remark}
The above proposition shows that the modal logic given by any presentation of $\mathbf L'$ has the bounded proof property in the sense of~\cite{bezh-ghil:bpp}.
\end{remark}
\smallskip

\paragraph{\textbf C}\label{par:c}%
We can now continue the reasoning begun in Paragraphs~\hyperref[par:a]{A}-\hyperref[par:b]{B}. Using Proposition~\ref{prop:Lpresinj}, we are able to repeat the proof of~\cite[Theorem~6.15]{kurz-petr:manysorted}, showing by induction that the $n$-step semantics 
\begin{equation}\label{eq:nstepsemantics}
(\mathbf L')^{n}\Two\to P'(T')^n 1
\end{equation}
is injective for all $n\in\mathbb{N}$, which implies completeness with respect to the Kripke semantics given by $T'$-coalgebras.

\medskip

\paragraph{\textbf D}\label{par:d}%
Finally, since $T'$ is an extension of $T$, we have $T'D\cong DT$, which implies together with $P'D\cong WP$ that 
$P'(T')^n 1 \cong WPT^n1$, so that \eqref{eq:nstepsemantics} also gives completeness with respect to the Kripke semantics given by $T$-coalgebras.

\begin{remark}
As in Corollary~\ref{cor:mainthm}, or rather dually, the fact that we have $T'D\cong DT$ means that the adjunction $C\dashv D:\Set\to\Pos$ lifts to an adjunction $\widetilde C\dashv \widetilde D:\Coalg(T)\to\Coalg(T')$ so that $\Coalg(T)$ is a full reflective subcategory of $\Coalg(T')$. In particular, $\widetilde D$ preserves limits and therefore `behaviours' as given by final coalgebras or the final sequence of $T$.
\end{remark}

%========================================================%

\section{Monotone predicate liftings}\label{sec:monotone_predicate_liftings}
 
\medskip\noindent
We show that the logic of the posetification $T'$ of $T$ coincides with the logic of all monotone predicate liftings of $T$. 

\medskip\noindent Recall from~\cite{Pattinson03,Schroder05} that a predicate lifting of arity $n$ for $T$ is a natural transformation
\[
\heartsuit : \Set({-},2^n)\to\Set(T{-},2)
\]
Using the (ordinary!) 
adjunction $D\dashv V:\Pos\to\Set$, a predicate lifting can be described as a natural transformation
\[
\heartsuit':\Pos(D{-},[Dn,\mathbbm{2}])\to\Pos(D T{-},\mathbbm{2})
\]
It is called {\em monotone\/} if each component is monotone (as a map between hom-posets). By Yoneda lemma, one can also identify a predicate lifting with a map $\heartsuit: T(2^n)\to 2$. Then the above simply says that $\heartsuit$ is monotone if for all  $\overline {a_1}\le \overline {a_2}:DX\to[Dn, \Two]$, we have that   $\overline{\heartsuit \circ Ta_1}\le \overline{\heartsuit \circ Ta_2}$, where  $\overline{f}:DX\to Y$ denotes the adjoint transpose of $f:X\to  VY$.

\medskip\noindent Consider now a locally monotone $\Pos$-functor $T'$ and a finite poset $ p$. By mimicking the above, we define a predicate lifting for $T'$ of arity $ p$ as being a $\Pos$-natural transformation  
$$\heartsuit':\Pos(-, [ p, \Two]) \to \Pos(T'-, \Two)$$
which again can be identified with $\heartsuit' \in \Pos(T'([p,\Two]), \Two)$.

\begin{theorem}
\label{thm:mon-lift}
Let $T$ be an endofunctor of $\Set$ and $T':\Pos\to \Pos$ its posetification. Then there is a bijection between the predicate liftings of $T'$ of discrete arity $Dn$ and the monotone predicate liftings of $T$ of arity $n$, for each finite $n$. 
\end{theorem}

\begin{proof}
 
Let $p$ be an arbitrary finite poset. 
Consider the composition of the two following monomorphisms:
\begin{equation}\label{predlift}
\Pos (T' ([p,\Two]),\Two)\to \Set (VT' ([p, \Two]),V \Two)\to \Set (T V ([p, \Two]),V\Two)
\end{equation}
The first arrow above is monic by faithfulness of $V$. The second one is also, as it is given by pre-composition with the natural epimorphism 
$\tau:T V \to V T'$ (the mate of the isomorphism $\alpha:DT\to T'D$ under the adjunction $D\dashv V$). The latter is indeed epic because for each poset $X$, $\tau_X$ is exactly the coinserter map $TX_0\to T'X$.  

In case the arity is discrete as $ p=D n$, notice that by $V([Dn, \Two])=2^n$, the right hand side of Equation~\eqref{predlift} is precisely $\Set(T(2^n), 2)$. A predicate lifting $\heartsuit'\in \Pos(T'([ Dn, \Two]), \Two)$ is then sent to 
\[
\heartsuit :=V \heartsuit' \circ \tau_{[Dn,\Two]}:T(2^n)\to 2
\]
Now, for $a:X\to 2^n$, easy diagram chasing shows that
\[
\overline{\heartsuit\circ Ta}=\heartsuit' \circ T'(\overline a) \circ \alpha_X
\]
hence the monotonicity of $\heartsuit$ follows. Thus the predicate liftings of $T'$ of discrete arity are among the monotone predicate liftings for $T$. 

\medskip\noindent To show the inverse correspondence, recall one more time that the posetification $T'$ is constructed as a coinserter (Theorem~\ref{thm:posetification}). Let $\heartsuit:T(2^n)\to 2$ be a predicate lifting for $T$. Then, from the universal property of coinserters, one can easily check that $\overline \heartsuit:DT(2^n)\to \Two$ factorizes to a predicate lifting for $T'$ of discrete arity, $T'([Dn,\Two])\to \Two$, if and only if $\heartsuit $ is monotone in the sense mentioned above. More in detail: let $X_0$ be the set $2^n=V[Dn, \Two]$; that is, the underlying set of the poset $[Dn, \Two]$, and $X_1$ the underlying set of the order on $[Dn, \Two]$, with projections denoted as usual $d^0,d^1:X_1 \to X_0$. Then with notations as above, one has $\overline{d^0} \leq \overline{d^1} $; thus if $\heartsuit $ is monotone, this entails $\overline {\heartsuit \circ Td^0}\leq \overline{\heartsuit\circ Td^1}$, thus $\overline{ \heartsuit}:DTX_0\to \Two $ factorizes in $\Pos$ to a predicate lifting for $T'$ of discrete arity $\heartsuit ':T'[Dn, \Two]\to \Two$.
\[
\xymatrix@C=7pc{T'DX_1 \ar[d]^{\alpha^{-1}_{X_1}} \ar@<-0.5ex>[r]_{T'Dd^1} \ar@<0.5ex>[r]^{T'Dd^0} & T'DX_0 \ar[dr]^{T'\epsilon_{[Dn,\Two]}} \ar[d]^{\alpha^{-1}_{X_0}} 
\\
DTX_1 \ar@<-0.5ex>[r]_{DTd^1} \ar@<0.5ex>[r]^{DTd^0} & DTX_0 \ar[r]^-e \ar[dr]_{\overline \heartsuit} & T' [Dn, \Two] \ar[d]^{\heartsuit '} \\
&&\Two }
\]
From the above diagram we have that $\heartsuit'\circ T'\epsilon_{[Dn,\Two]} \circ \alpha_{V[Dn,\Two]} = \overline \heartsuit$, thus we see we can recover the original monotone predicate lifting for $T$:
\[
V\heartsuit '\circ \tau_{[Dn, \Two]} = V \heartsuit' \circ VT'\epsilon_{[Dn,\Two]} \circ V\alpha_{V[Dn,\Two]} = V \overline \heartsuit = \heartsuit
\]
Finally, note that we have used the assumption that $T'$ is the posetification of $T$ in order to have an extension such that $TV\to VT'$ is epi.
\end{proof}

\begin{remark}
The theorem should be seen in the light of \cite[Theorem 4.16]{kkv:aiml} saying that for a finitary and embedding preserving functor $T':\Pos\to\Pos$ the logic $\mathbf L'$ of (necessarily monotone) predicate liftings is expressive. We also would like to recall \cite[Corollary~6.9]{kurz-leal:mfps09-j} which describes an expressive and monotone subset of all predicates liftings for any finitary weak-pullback preserving $T:\Set\to\Set$. 
\end{remark}

%========================================================%

\section{Conclusions}%\label{sec:conclusions}

In the area of semantics of programming languages one encounters a wide variety of base categories including various metric spaces and various (complete) partial orders. It would be of interest to draw the landscape of these different categories together with a toolkit connecting them. This paper can be seen as a rudimentary effort in this direction. Indeed, we relate systems and their logics across the morphism of connections 
\[
(S\dashv P:\Set^\op\to\BA)\ \longrightarrow\  (S'\dashv P':\Pos^\op\to\DL)
\]
Moreover, we transfer functors along this morphism via left Kan-extensions and characterize the functors that arise in that way as those preserving certain classes of colimits. Finally, we show how results about modal logics can be derived from such a framework. 
It will be interesting to explore whether similar techniques apply to more sophisticated domains than $\Set$ and $\Pos$.

%========================================================%

\end{document}